\documentclass[11pt,oneside,english]{amsart}
\usepackage[T1]{fontenc}
\usepackage{lmodern}
\usepackage[latin9]{inputenc}
\usepackage{enumerate}
\usepackage{verbatim}
\usepackage{mathtools}
\usepackage{amstext}
\usepackage{amsthm}
\usepackage{amssymb}
\usepackage{linegoal}
\usepackage{stmaryrd}
\usepackage[top=28mm,right=30mm,bottom=28mm,left=30mm]{geometry}

\makeatletter
\numberwithin{equation}{section}
\numberwithin{figure}{section}
\theoremstyle{plain}
\newtheorem{thm}{\protect\theoremname}[section]
\theoremstyle{plain}
\newtheorem{cor}[thm]{\protect\corollaryname}
\theoremstyle{remark}
\newtheorem{rem}[thm]{\protect\remarkname}
\theoremstyle{plain}
\newtheorem{lem}[thm]{\protect\lemmaname}
\theoremstyle{plain}
\newtheorem{prop}[thm]{\protect\propositionname}
\theoremstyle{defn}
\newtheorem{defn}[thm]{Definition}

\usepackage[backref=page,colorlinks,citecolor=blue,bookmarks=true]{hyperref}\hypersetup{linkcolor=blue,  citecolor=blue, urlcolor=blue}
\usepackage{amsrefs}

\usepackage{microtype}

\let\originalleft\left
\let\originalright\right
\renewcommand{\left}{\mathopen{}\mathclose\bgroup\originalleft}
\renewcommand{\right}{\aftergroup\egroup\originalright}

\date{}

\makeatother

\usepackage{babel}
\providecommand{\corollaryname}{Corollary}
\providecommand{\lemmaname}{Lemma}
\providecommand{\propositionname}{Proposition}
\providecommand{\remarkname}{Remark}
\providecommand{\theoremname}{Theorem}

\newcommand{\eps}{\varepsilon}

\newcommand{\frakp}{\mathfrak{p}}%
\newcommand{\fraka}{\mathfrak{a}}%
\newcommand{\frakP}{\mathfrak{P}}%

\newcommand{\bF}{\mathbf{F}}

\newcommand{\R}{\mathbf{R}}

\renewcommand{\AA}{\mathbb{A}}

\newcommand{\CC}{\mathbb{C}}

\newcommand{\EE}{\mathbb{E}}
\newcommand{\FF}{\mathbb{F}}
\newcommand{\GG}{\mathbb{G}}
\newcommand{\HH}{\mathbb{H}}

\newcommand{\KK}{\mathbb{K}}
\newcommand{\LL}{\mathbb{L}}

\newcommand{\NN}{\mathbb{N}}

\newcommand{\PP}{\mathbb{P}}
\newcommand{\QQ}{\mathbb{Q}}

\newcommand{\UU}{\mathbb{U}}

\newcommand{\ZZ}{\mathbb{Z}}

\newcommand{\calB}{\mathcal{B}}
\newcommand{\calC}{\mathcal{C}}

\newcommand{\calE}{\mathcal{E}}

\newcommand{\calN}{\mathcal{N}}
\newcommand{\calO}{\mathcal{O}}
\newcommand{\calP}{\mathcal{P}}

\DeclareMathOperator{\SL}{SL}%
\DeclareMathOperator{\PGU}{PGU}%
\DeclareMathOperator{\SU}{SU}%
\DeclareMathOperator{\PGL}{PGL}%
\DeclareMathOperator{\Gal}{Gal}%
\DeclareMathOperator{\Frob}{Frob}%
\DeclareMathOperator{\spec}{spec}%
\DeclareMathOperator{\nm}{Nm}%
\DeclareMathOperator{\GL}{GL}%
\DeclareMathOperator{\Res}{Res}%
\DeclareMathOperator{\adj}{adj}%
\DeclareMathOperator{\Spec}{Spec}%
\DeclareMathOperator{\Cay}{Cay}%
\DeclareMathOperator{\rk}{rk}%

\DeclareMathOperator{\Aut}{Aut}%
\DeclareMathOperator{\supp}{supp}%
\DeclareMathOperator{\Out}{Out}%
\DeclareMathOperator{\id}{id}%
\DeclareMathOperator{\diam}{diam}%
\DeclareMathOperator{\al}{al}%
\DeclareMathOperator{\der}{der}%
\DeclareMathOperator{\distinct}{distinct}%
\DeclareMathOperator{\Unif}{Unif}%
\DeclareMathOperator{\mspec}{mspec}

\DeclareMathOperator{\PSU}{PSU}
\DeclareMathOperator{\op}{op}

\begin{document}

\author[O.\ Becker]{Oren Becker} \address{Oren Becker\hfill\break 	Department of Pure Mathematics and Mathematical Statistics\hfill\break 	Centre for Mathematical Sciences\hfill\break 	Wilberforce Road, Cambridge CB3 0WA, United Kingdom} \email{oren.becker@gmail.com}
\author[E.\ Breuillard]{Emmanuel Breuillard} \address{Emmanuel Breuillard\hfill\break 	Mathematical Institute \hfill\break Oxford OX1 3LB, United Kingdom} \email{emmanuel.breuillard@maths.ox.ac.uk}
\title{}
\title[Uniform expansion in finite groups of Lie type]{Uniform expansion in finite groups of Lie type}


\maketitle

{\centering \it Dedicated with admiration and affection to Alex Lubotzky on his 70th birthday.\par}

\begin{abstract}
We prove that finite simple groups $G(p)$ of bounded rank and with $p$ prime have uniform expansion, that is, the family of all the Cayley graphs forms a family of (two-sided) expanders, except perhaps when $p$ belongs to a small family of exceptional primes.
Furthermore, for all prime powers $q$, the set of possible exceptions to uniform expansion of $G(q)$ is shown to have ``dimension zero''. We also extend these results to semisimple and perfect algebraic groups.
\end{abstract}

\tableofcontents

\section{Introduction}

In \cite{bourgain-gamburd} J. Bourgain and A. Gamburd introduced a new method to obtain spectral gaps for Cayley graphs of finite groups. Using it they showed that every Cayley graph of $\SL_2(p)$, $p$ prime, with girth comparable to $\log p$ is an expander graph. It follows \cite[Theorem 1.2]{bourgain-gamburd} that random Cayley graphs of $\SL_2(p)$ are expanders and that so are Cayley graphs obtained from the reduction modulo $p$ of a fixed Zariski-dense subgroup of $\SL(2,\ZZ)$ (super strong approximation).

In the second-named author's ICM address \cite{breuillard-icm}, it was conjectured that finite simple (or quasi-simple) groups of Lie type $G(q)$  of bounded rank have \emph{uniform expansion}. This means that the entire family of all Cayley graphs of these groups should form a family of expanders. In other words, it posits the existence of a gap $\eps=\eps(r)>0$ such that \begin{equation}\label{expand}|AB|\ge (1+\eps)|B|\end{equation} for all subsets $A,B \subset G(q)$ such that $|B|\leq |G(q)|/2$ and $A$ generates and contains the identity and all finite simple groups $G(q)$ of rank at most $r$. The question of uniform expansion was first raised by Lubotzky and Weiss in \cite{lubotzky-weiss}. (See also the recent survey \cite{breuillard-lubotzky} for related questions and the book \cite{lubotzky} and survey \cite{wigderson-et-al} for background on expanders.) In \cite{breuillard-gamburd} A. Gamburd and the second named author offered some evidence towards this conjecture. They showed that, except for some zero-density set of primes $p$, the family of all Cayley graphs of $\SL_2(p)$ forms a family of expander graphs. This provided the first example of an infinite family of finite groups admitting uniform expansion, answering a problem formulated in  \cite[\S 5]{lubotzky-weiss}.

At the same time, Green, Guralnick, Tao and the second named author \cite{bggt} extended the aforementioned result of Bourgain-Gamburd \cite{bourgain-gamburd} about random Cayley graphs of $\SL_2(p)$ to all finite simple groups of Lie type $G(q)$, $q$ a power of $p$. It was shown in \cite{bggt} that there is a power saving $\delta>0$ depending only on the rank $r$ of $G$ such that all but $|G(q)|^{2-\delta}$ pairs $(a,b) \in G(q)$ give rise to expander Cayley graphs of $G(q)$. 

In the present paper we improve these results in two ways. Our first purpose will be to extend \cite{breuillard-gamburd}  to all other finite simple groups of Lie type (see Theorems \ref{expmain} and \ref{perfectaddon} for extensions to semisimple and perfect algebraic groups):





\begin{thm}[uniform expansion mod $p$]\label{cro-refor}
Fix $r, m \in \NN$. The family of all finite simple groups $G(p^k)$ of
Lie type with rank at most $r$ and $k \leq m$ has $\eps$-uniform
expansion, $\eps > 0$, as long as a (possibly empty) density-zero set
$\calP$ of primes $p$ is excluded: for every $\delta \in (0,1)$ we can
ensure $\#\{p \in \calP \mid p \leq X\} \leq X^{\delta}$ for all $X$, with
$\eps = \eps(r, m, \delta)>0$.
\end{thm}

In going from $\SL_2$ to arbitrary simple $G$ we managed to overcome a well-known difficulty in the field concerning non-concentration on subgroups (see e.g. \cite{golsefidy-srinivas1}). Theorem \ref{cro-refor} is also a key ingredient in our forthcoming work on character varieties of random groups with P. Varj\'u \cite{becker-breuillard-varju}.

The second main result of this paper is an improvement on \cite{bggt}, where we can now assert that exceptions to the uniform expansion conjecture have ``\emph{dimension zero}'': 

\begin{thm}[dimension zero theorem]\label{dim0} For every positive integer $r$, every $\delta>0$, there is $\eps=\eps(r,\delta)>0$ such that for every prime power $q$ and every finite simple group of Lie type $G(q)$ of rank at most $r$, all but at most $O_{\delta,r}(q^\delta)$ conjugacy classes of generating pairs $a,b \in G(q)$, give rise to $\eps$-expander Cayley graphs, i.e. \eqref{expand} holds for $A=\{1, a,a^{-1},b,b^{-1}\}$ with $\eps=\eps(r,\delta)>0$.
\end{thm}

This is new already for $\SL_2$.  We count conjugacy classes of pairs, namely pairs up to the diagonal action $g\cdot(a,b)=(gag^{-1},gbg^{-1})$. This is natural since property \eqref{expand} is invariant under this action. Theorem \ref{dim0} also holds for $m$-tuples for any $m\geq 2$, see the more precise formulation in Theorem \ref{expmaindim0}.

\vspace{.5cm}

\noindent {\bf Remarks.} (a) The possible existence of exceptional primes $p$ for which uniform expansion might fail in $G(p)$ and that of the fewer than $q^{\delta}$ possible classes of exceptional pairs in $G(q)$ that might lead to non-expander Cayley graphs for $G(q)$ cannot be ruled out by our method alone. In the case of exceptional primes, the reason has to do with our use of effective-arithmetic-nullstellensatz-flavoured results (see Section \ref{bad-red} and \cite{BeckerBreuillardReductions}):
While we would wish for bad reduction associated to word maps of length $\ell$ to be confined to primes of size at most exponential in $\ell$, we cannot rule out bad reduction up to the super-exponential bound $\exp(\ell^{O_r(1)})$.
Although such bounds are sharp for general polynomials, so that no general-purpose argument can do better, it remains possible that the specific structure of the word maps arising in our group-theoretic context, and of their reduction mod $p$, could be exploited to improve them.
Similarly in Theorem \ref{dim0}, we have only polynomial control $\ell^{O_r(1)}$ on the degree of the potential exceptions (the Bezout bound), while linear control would be needed to rule them out. 

We thus have to allow for the potential existence of a few pairs with bad reduction, or  \emph{``{\bf scars}''}, namely pairs $(a,b)$ that although they do generate $G(q)$ behave as if they belonged to a proper subgroup.  

Without discarding any exceptional pairs in Theorem \ref{dim0}, we can get exponential non-concentration for all generating sets up to length $(\log |G|)^c$ for some $c=c(r)>0$. This however is not enough for proving global expansion.

(b) An explicit lower bound on the expansion gap $\eps>0$ can be traced from our arguments (see e.g. Remark \ref{explicit} for Theorem \ref{cro-refor}). 
Since this expression is likely not sharp we do not dwell on this. 

(c) Expansion for a family of Cayley graphs $\mathcal{G}$ implies fast (logarithmic) mixing of the simple random walk $(S_n)_{n\ge 0}$ on the graph (it implies double-sided expansion, see \cite[Appendix E]{bggt}, \cite{Biswas} and \cite{alex-private}). Conversely, logarithmic mixing of the simple random walk on a Cayley graph implies expansion (see \cite[Lemma 3.3]{breuillard-standrews}). This idea is at the heart of the Bourgain-Gamburd method for establishing spectral gaps in Cayley graphs (see \cite[Proposition 3.1]{bggt}, \cite{breuillard-standrews}, \cite{tao-expansion-book}). There are three phases towards mixing. The last two phases are fully established for the groups we consider: they follow from the classification of approximate groups \cite{helfgott, pyber-szabo, breuillard-green-tao-linear} and quasi-randomness \cite{landazuri-seitz, sarnak-xue, gowers}. We thus concentrate on the first phase (the only stage that depends on the choice of generating set $S$), which aims to prove exponential non-concentration of the walk on subgroups up to $n \simeq \log |G|$.

(d) The proof of Theorem \ref{cro-refor} proceeds  by reduction modulo primes of the first order logic formulae that express the non-concentration on subgroups. This idea of reduction is similar in spirit to the argument from \cite{breuillard-gamburd}, where the effective arithmetic nullstellensatz was used to handle the case of $\SL_{2}\left(p\right)$. To effect this we crucially rely on the uniform anti-concentration bounds for random walks in characteristic zero that we obtained in our previous work \cite{becker-breuillard}. These provide a uniform exponential bound on the hitting probability of  random walks on algebraic subvarieties in algebraic groups over infinite fields. The proof of the anti-concentration bounds in \cite{becker-breuillard}  is based on diophantine considerations, in particular the Height gap theorem from \cite{breuillard-annals}. Conversely, uniform expansion for $G(p)$ in combination with the Lang-Weil estimates implies these anti-concentration bounds. So the anti-concentration bounds from \cite{becker-breuillard} are a \emph{necessary step} towards Theorem \ref{cro-refor}.  They allow us to show that the random walk does not \emph{concentrate on subgroups}, which is a well-known difficulty that we manage to overcome here (see e.g. \cite{golsefidy-srinivas} where this difficulty is reduced to the case of simple groups).

(e) Another new feature that we bring in the proofs of Theorems \ref{cro-refor} and \ref{dim0} is the consideration of the number of \emph{types} cut out by the above-mentioned first order logic formulae. We prove an exponential bound on them, where the naive bound is super-exponential. For Theorem \ref{cro-refor} we also require a polynomial bound for the discriminant of bad reduction for affine varieties defined over $\ZZ$, established in \cite{BeckerBreuillardReductions} using explicit Groebner basis estimates. This will be the topic of Section \ref{bad-red}.


(f) The proof of Theorem \ref{dim0}  will further use the positive characteristic version of the same anti-concentration bounds, which we also established in \cite{becker-breuillard}. They have the crucial feature of being  uniform regardless of the characteristic of the ambient field. Our approach  allows us to handle subfield subgroups uniformly (without any special subcase as in \cite{bggt, bgt-suzuki} where triality groups ${}^{3}D_4(q)$ or Suzuki-Ree groups had to be handled separately) which was a key difficulty in previous attempts at generalizing super-strong approximation to positive characteristic.

(g) In  \cite{salehi-varju}, Golsefidy and Varj\'u  established super strong approximation in characteristic zero. This asserts that finitely generated Zariski-dense subgroups of perfect algebraic groups admit expander quotients modulo squarefree integers (see \cite[Window 9]{LubotzkySegal}, \cite{rapinchuk}, \cite{pink} for background on ordinary strong approximation and  \cite{salehi-sarnak, bourgain-gamburd-sarnak} for applications to number theory (affine sieve); see also \cite{breuillard-standrews} for an overview and \cite{avni-gelander} for a recent improvement).  It is an open problem to extend their result to positive characteristic. Theorem \ref{dim0} and the method presented here can be used to prove super-strong approximation in positive characteristic modulo prime ideals of prime degree only, or modulo ``most prime ideals''. But getting all prime ideals requires a new idea.

\vspace{.5cm}


\noindent \emph{Conventions.} We use Vinogradov's notation: $f\ll g$ stands for $f\leq Cg$, where
$C$ is an absolute constant. We write $f\ll_{X}g$
when $C$ is allowed to depend only on $X$. And $f\simeq_{X}g$ means both $f\ll_{X}g$ and $g\ll_{X}f$. Also for a positive integer $r$ we write $[r]$ for the set $\{1,\ldots,r\}$. The notation $\PP_{s\sim\Unif\left(S\right)}$ means that we are considering the probability of an event for $s$ chosen uniformly at random in the set $S$.  
We denote by $K^{\al}$ an algebraic closure of the field $K$.

\section{Good reduction and exponential bound on types}\label{bad-red}
In this section, we prove Proposition \ref{prop:types}, a technical result that will be needed for the proof of Theorem \ref{cro-refor}. The point is to derive bounds for the largest prime with bad reduction for a family of $\ZZ$-schemes. The bounds are given in terms of the degrees and heights of generators of the corresponding ideal. We begin with a general discussion regarding $\ZZ$-schemes and their reduction modulo primes, and then recall the main results of our companion paper \cite{BeckerBreuillardReductions}.

Proposition \ref{prop:types}  will only be required in the special case where the ground field is $\QQ$. Nevertheless, even in this special case, the proof of Proposition \ref{prop:types} will need to proceed by considering more general number fields. Thus, we must state the required tools, namely Theorem \ref{thm:good-reductions} and Proposition \ref{prop:inclusion-preserved-mod-p}, in the generality of number fields, and we shall prove Proposition \ref{prop:types} over any number field.

We begin with recalling some terminology regarding schemes over rings of integers of number fields. We write $\spec R$ (resp. $\mspec R$) for the set of prime (resp. maximal) ideals of a commutative unital ring $R$. Recall that if $R$ is the ring of integers of a number field then $\mspec R=(\spec R)\setminus\{(0)\}$.

For $\alpha\in\QQ^{\al}$ (resp. $\FF_p(t)^{\al}$), we write $h(\alpha)$ for the logarithmic Weil height of $\alpha$, that is $$h(\alpha)=\frac{1}{d}\sum_{v \in V_K} n_v \log^+|\alpha|_v,$$ where $K$ is any finite extension of $\QQ$ (resp. $\FF_p(t)$) containing $\alpha$, $d$ its degree, $V_K$ the set of places (i.e. equivalence classes of absolute values) of $K$ and $n_v=[K_v:\QQ_v]$ the degree of the local extension $K_v$, the completion of $K$ with respect to $v$. We refer the reader to \cite{BombieriGubler2006} for background on heights. 
For polynomials $f_1,\dotsc,f_s$ with coefficients in $\QQ^{\al}$ (resp. $\FF_p(t)^{\al}$)
we write $h_{\max}(f_1,\dotsc,f_s)$ for $\max\{h(\alpha)\mid \alpha\in\calC\}$, where $\calC$ is the set of coefficients of $f_1,\dotsc,f_s$.

In what follows, we handle the general case of global fields (that is finite extensions of either $\QQ$ or $\FF_p(t)$) and we use the following shorthand notation: we denote by $F_0$ either $\QQ$ or $\FF_p(t)$ and we let $F_0^{\al}$ be an algebraic closure. 

For a subset $S$ of the polynomial ring $F_0^{\al}[T_1,\dotsc,T_r]$, we write $V(S)$ for the affine algebraic set of all points ${\bf x}\in (F_0^{\al})^r$ such that $f({\bf x})=0$ for all $f\in S$. For a global field $F$ (i.e. a finite extension of $F_0$) with ring of integers $\calO_F$
and a prime ideal $\frakp\in\mspec\calO_F$,
we write $S_{\frakp}$ for the image of $S\cap\calO_{F}[T_1,\dotsc,T_r]$ in $(\calO_{F}/\frakp)[T_1,\dotsc,T_r]$ after reduction modulo $\frakp$.
For a nonzero ideal $\fraka$ of $\calO_F$ we write $\calN\fraka$ for the absolute norm of $\fraka$, i.e. $\calN\fraka=[\calO_F:\fraka]$.

In this section from now on we will work in characteristic zero, except for Lemmas \ref{lem:component-of-few-eqs-1} and \ref{lem:deg-height--word-eqs}. These lemmas work in all characteristics and will be used in Section \ref{dzerosec} in positive characteristic. 

We denote the field of definition of an algebraic subset $X$ of $(\QQ^{\al})^{r}$ by $F_X$. This is the fixed field of the subgroup of Galois automorphisms $\sigma$ that preserve $X$ setwise, i.e. $\sigma X=X$.
Similarly, the field of definition $F_I$ of an ideal $I$ of $\QQ^{\al}[T_1,\dotsc,T_r]$ is the fixed field of the subgroup of Galois automorphisms $\sigma$ that preserve $I$ setwise, i.e. $\sigma I=I$ (and is also the smallest subfield of $\QQ^{\al}$ over which $I$ admits a generating set). Note that $F_X=F_{I(X)}$ where $I(X)$ is the radical ideal consisting of all polynomials in $\QQ^{\al}[T_1,\dotsc,T_r]$ that vanish on $X$. Also note that $F_{\sqrt{I}}\subset F_I$ and that while $F_{X}$ is contained in the compositum of fields of definitions of the irreducible components of $X$, it is possible for $F_I$ not to be contained in the compositum of the fields of definitions of the components of $V(I)$ if $I$ is not radical.

An ideal $J$ of $\QQ^{\al}[T_1,\dotsc,T_r]$ is said to have \emph{complexity} at most $(d_0,h_0)$ for integers $d_0,h_0 \ge 1$ if
there are $g_1,\dotsc,g_s\in F_J[T_1,\dotsc,T_r]$ with $\max_{i=1}^{s} \deg g_i \leq d_0$
and
$h_{\max}(g_1,\dotsc,g_s)\leq h_0$ that generate $J$.
The requirement that the coefficients of the $g_i$ lie in $F_J$ is not very limiting. Indeed, if $J$ is generated by polynomials $g_1,\dotsc,g_s\in F_0^{\al}[T_1,\dotsc,T_r]$ of degree at most $d_0$ with $h_{\max}(g_1,\dotsc,g_s)\leq h_0$ then $J$ has complexity at most $(d_0^{O_r(1)}, O_r(d_0^{O_r(1)})h_0)$ (this follows from \cite[Proposition 4.17]{BeckerBreuillardReductions}, noting that the coefficients of every element of a reduced Groebner basis of $J$ lie in $F_J$).

The following summarizes the main result from \cite{BeckerBreuillardReductions}.

\begin{thm}
\label{thm:good-reductions}\cite{BeckerBreuillardReductions}
Let $I$ be an ideal of
$\QQ^{\al}[T_1,\dotsc,T_r]$
generated by polynomials
$f_{1},\dotsc,f_{s}$.
Write
$d_0=\max_{i=1}^{s}\deg f_i$
and
$h_0=\max\{h_{\max}(f_1,\dotsc,f_s),1\}$.
Write $L$ for the compositum of the fields of definition of the irreducible components $X_1,\ldots,X_t$ of $V(I)$.
Then there are ideals
$I_1,\dotsc,I_t$
of
$\QQ^{\al}[T_1,\dotsc,T_r]$,
each of complexity at most
$(d_0^{O_{r}(1)},O_r(d_0^{O_r(1)} h_0))$,
and a positive integer
$\Delta\in\NN$,
$\log\Delta\ll_r [F_I:\QQ]^{O(1)}d_0^{O_{r}(1)} h_0$,
such that:
\begin{enumerate}
    \item 
    $V(I_1),\dotsc,V(I_t)$ are the distinct irreducible components of $V(I)$.

    \item
    $F_{I_{j}}=F_{X_{j}}$ and
    $[F_{X_{j}}F_{I}:F_{I}]\leq d_0^{O_{r}(1)}$
    for each $1\leq j\leq t$.

    \item 
    For every $\frakp\in\mspec\calO_{L F_I}$
    such that
    $\gcd(\calN \frakp,\Delta)=1$,
    the distinct irreducible components of $V(I_{\frakp})$ are
    $V((I_1)_\frakp),\dotsc,V((I_t)_\frakp)$,
    and $\dim V((I_j)_{\frakp}) = \dim V(I_j)$ for all $1\leq j\leq t$.
    
\end{enumerate}
\end{thm}
\begin{proof}
The statement of this theorem combines the statements of (slightly weakened version of) Theorems 1.1 and 1.2 of \cite{BeckerBreuillardReductions}, dropping any mention of the Frobenius actions since they are not needed in the present paper.
\end{proof}

\begin{prop}
\label{prop:inclusion-preserved-mod-p}
Let $I$ and $J$ be
ideals of
$\QQ^{\al}[T_1,\dotsc,T_r]$,
 with
complexity at most $\left(d_{0},h_{0}\right)$.
Then there is $\Delta\in\NN$, $\log\Delta\ll_r [F_{I}F_{J}:\QQ]^{O(1)}d_{0}^{O_{r}\left(1\right)}h_{0}$,
such that for every $\frakp\in\mspec\calO_{F_{I}F_{J}}$  with $\gcd(\calN\frakp,\Delta)=1$ we have  $V(I_{\frakp})\subseteq V(J_{\frakp})$ if and only if $V(I)\subseteq V(J)$.
\end{prop}
\begin{proof}
Assume that $V(I)\subset V(J)$.
By \cite[Proposition 4.17]{BeckerBreuillardReductions}, there is a Groebner basis $G\subset \calO_{F_I}[T_1,\dotsc,T_r]$ for $I$ (for any fixed monomial ordering) such that $\max_{g\in G}\deg g\ll_r d_0^{O_r(1)}$, $|G| \ll_r d_0^{O_r(1)}$ and
$\max_{g\in G}h_{\max}(g)\ll_r [F_{I}:\QQ]d_0^{O_r(1)}h_0$.
Let $\alpha$ be the product of the leading coefficients of the elements of $G$.
Set
$\Delta_1=|\nm_{F_{I}F_{J}/\QQ}\alpha|$.
Then
$\log\Delta_1\leq [F_{I}F_{J}:\QQ]h(\alpha)\ll_r [F_{I}F_{J}:\QQ]^2d_0^{O_r(1)}h_0$.
Take $\frakp\in\mspec\calO_{F_{I}F_{J}}$ such that $\gcd(\calN\frakp,\Delta_1)=1$.
Then $\alpha\notin\frakp$ and so,
by \cite[Lemma 6.5]{BeckerBreuillardReductions},
$V(I_{\frakp})\subseteq V(J_{\frakp})$.  Conversely, \cite[Lemma 6.9]{BeckerBreuillardReductions} gives an integer $\Delta_2$ satisfying similar bounds for which $V(I_{\frakp})\nsubseteq V(J_{\frakp})$ provided $V(I)\nsubseteq V(J)$ and $\gcd(\calN\frakp,\Delta_2)=1$. Setting $\Delta:=\Delta_1\Delta_2$ yields the desired conclusion.
\end{proof}

The next lemma will be essential to the proof of Theorems \ref{cro-refor} and \ref{dim0}. It gives good control on the number of \emph{types}, that is, geometrically irreducible components, cut out by a large finite family of polynomials.

\begin{lem}
\label{lem:component-of-few-eqs-1}Let $k$ be  an algebraically closed field and let $E$ be some family of polynomials in $k\left[T_{1},\dotsc,T_{r}\right]$.  Then, for every irreducible component $C$ of $V\left(E\right)$,
there are distinct $g_{1},\dotsc,g_{s}\in E$, $s\leq r-\dim C$,
and $C=C_{s}\subsetneq C_{s-1}\subsetneq\cdots\subsetneq C_{0}=\AA_{k}^{r}$,
such that $C_{i}$ is an irreducible component of $C_{i-1}\cap V\left(g_{i}\right)$
for each $1\leq i\leq s$. In particular, if $d_0=\max_{f \in E}  \deg f$, then there are at most $2|E|^r d_0^{r(r+1)/2}$ closed irreducible subvarieties of $\AA_{k}^{r}$ that arise as irreducible components of $V(E')$ for some $E'\subset E$.
\end{lem}

\begin{proof}
Let $C_{0}=\AA_{k}^{r}$. We define, by induction, $g_{1},\dotsc,g_{s}\in E$
and $C_{1},\dots,C_{s}$, such that $C_{i}\subsetneq C_{i-1}$, $C_{i}$
is an irreducible component of $C_{i-1}\cap V\left(g_{i}\right)$,
and $C_{i}$ contains $C$.

For $i\geq1$, if $C_{i-1}=C$, we halt with $s=i-1$. Otherwise,
$C_{i-1}$ is irreducible and contains $C$ properly. So $C_{i-1}\nsubseteq V\left(E\right)$,
and thus there is $g_{i}\in E$ such that $C_{i-1}\nsubseteq V\left(g_{i}\right)$.
Let $C_{i}$ be an irreducible component of $C_{i-1}\cap V\left(g_{i}\right)$
that contains $C$. The process eventually halts with $s\leq r-\dim C$
because each $C_{i}$ is
irreducible and $C_{i}\subsetneq C_{i-1}$.

By the above, every irreducible component $C$ of every $V(E')$, $E'\subset E$, arises from $g_1,\cdots,g_s \in E'$ and $C_s\subsetneq\dotsc\subsetneq C_0$ as above. 
For fixed $g_1,\dotsc,g_s \in E$, there are at most $d_0^{1+2+\cdots+s}$ ways to produce a descending chain $C_0,\dotsc,C_s$ as above because once $C_{i-1}$ is chosen there are at most $d_0^i$ possible choices for $C_i$ since the degree of $C_{i-1}\cap V(g_i)$ is at most $d_0^i$ by Bezout's theorem.

\end{proof}

\begin{prop}[Bad reduction control]\label{prop:types} Let $\calE\subseteq\calO_F\left[T_{1},\dotsc,T_{r}\right]$
be a finite collection of polynomials,
$F$ a number field.
Write
$d_0=\max_{f\in\calE}
\deg f$
and
$h_0=\max_{f\in\calE}
\max\{h_{\max}(f),1\}$.
Then there is $\Delta\in\NN$,
$\log\Delta\ll_r |\calE|^{r+1}[F:\QQ]^{O(1)}d_0^{O_{r}\left(1\right)}h_0$, such that
for every  prime ideal $\frakp$
of $\calO_F$ with $\gcd(\calN\frakp,\Delta)=1$, and every pair of subcollections $E^1, E^2\subseteq\calE$, $V\left(E^1\right)\subseteq V\left(E^2\right)$ if and only if $V\left(E^1_{\frakp}\right)\subseteq V\left(E^2_{\frakp}\right)$.
\end{prop}

\begin{rem}
A direct application of Proposition \ref{prop:inclusion-preserved-mod-p}
to each subset $E$ of $\calE$ would lead to the proof of a weak version of Proposition \ref{prop:types}, where
the $|\calE|^{r+1}$ factor in the bound on $\log\Delta$ is replaced
by $2^{|\calE|}$. This bound is too weak for our purposes however, and the improved bound offered by Proposition
\ref{prop:types} via Lemma \ref{lem:component-of-few-eqs-1} will be crucial for our application to uniform spectral
gaps.
\end{rem}

\begin{proof}[Proof of Proposition \ref{prop:types}]
We shall define integers
$\Delta_{1},\Delta_{2} \in\NN$ with
\begin{equation}
\log\Delta_{1}\Delta_{2}\ll_r |\calE|^{r+1}[F:\QQ]^{O(1)}d_0^{O_{r}\left(1\right)}h_0\,\,\text{,}\label{eq:types-proof-bound-on-deltas}
\end{equation}
and a collection $\calB$ of ideals $J$
of $\QQ^{\al}[T_1,\dotsc,T_r]$ such that each geometrically irreducible component of $V(E)$ for each $E \subset \calE$  is of the form $V(J)$ for $J \in \calB$ and such that, if $K$ denotes the compositum of $\{F\}\cup\{F_{J}\mid J\in\calB\}$, then for every $\frakP\in\mspec\calO_{K}$ and $E\subseteq\calE$:
\begin{enumerate}
\item If $\gcd(\calN\frakP,\Delta_1)=1$, then each geometrically irreducible component of $V(E_{\frakP})$ is $V(J_{\frakP})$ for some  $V(J)$, $J\in \calB$, and 
\item If $\gcd(\calN\frakP,\Delta_2)=1$ and $J\in\calB$,
 then $V(J_{\frakP})\subseteq V\left(E_{\frakP}\right)$
if and only if $V(J)\subseteq V\left(E\right)$.
\end{enumerate}

This will prove the proposition with $\Delta=\Delta_1\Delta_2$. Indeed, if $\frakp\in\mspec\calO_F$ has  $\gcd(\calN\frakp,\Delta)=1$, we may take  $\frakP\in\mspec \calO_K$ a prime ideal lying over $\frakp$.
Then $\gcd(\calN\frakP,\Delta)=1$. 
To prove the equivalence between $V(E^1) \subset V(E^2)$ and $V(E^1_{\frakp}) \subset V(E^2_{\frakp})$
 we argue componentwise as follows.
 Consider a component of $V(E^1_{\frakp})$,
 which by (1) must be of the form $V(J_{\frakP})$ for some $J$ in $\mathcal{B}$. By (2), since $V(J_{\frakP}) \subset V(E^1_{\frakP})$, we know that $V(J)\subset V(E^1)$. If $V(E^1)\subset V(E^2)$, then $V(J)\subset V(E^2)$ as well, and by (2), $V(J_{\frakP})\subset V(E^2_{\frakP})$. This shows the forward implication. To see the converse, consider a component of $V(E^1)$. It is of the form $V(J)$ for some $J$ in $\calB$. By (2), $V(J_{\frakP}) \subset V(E^1_{\frakP})$, hence $V(J_{\frakP}) \subset V(E^2_{\frakP})$ assuming $V(E^1_{\frakp}) \subset V(E^2_{\frakp})$. By (2) again $V(J)$ lies in $V(E^2)$. This proves the proposition.

It remains to define $\calB$, $\Delta_{1}$, $\Delta_{2}$ satisfying (1)-(2) and \eqref{eq:types-proof-bound-on-deltas}. Write $\calE_{\distinct}^{\leq r}$ for the set of all tuples $\left(g_{1},\dotsc,g_{s}\right)\in\calE^{s}$
such that $g_{1},\dotsc,g_{s}$ are distinct and $s\leq r$. Then
$|\calE_{\distinct}^{\leq r}|\leq|\calE|^{r}$.
Fix $G=\left(g_{1},\dotsc,g_{s}\right)\in\calE_{\distinct}^{\leq r}$,
$s\leq r$.
We shall define a rooted tree $T_{G}^{\infty}$,
where each leaf is at
distance $s$ from the root. For each vertex $x$ of $T_{G}^{\infty}$
we shall associate a closed subvariety $V_{x}^{\infty}$
of $\AA_{\QQ^{\al}}^{r}$.

The tree $T_{G}^{\infty}$ and the varieties $V_{x}^{\infty}$, where
$x$ runs over the set of vertices of $T_{G}^{\infty}$, are defined
level-by-level as follows. For the root $x_{0}$ of $T_{G}^{\infty}$,
set $V_{x_{0}}^{\infty}=\AA_{\QQ^{\al}}^{r}$. For a vertex $x$
of $T_{G}^{\infty}$, at distance $0\leq\ell\leq s-1$ from the root,
the set of children of $x$ is in bijective correspondence with the
set of irreducible components of $V_{x}^{\infty}\cap V\left(g_{\ell+1}\right)$.
 Let $K_{G}$ be the compositum of $\{F_{V_x^{\infty}}\}_{x}$,
where $x$ runs over the vertices of $T^{\infty}_{G}$.

Let $K$ be the compositum of $\{F\}\cup\{K_{G}\}_{G}$, where $G$ runs over $\calE_{\distinct}^{\leq r}$.
For each $G\in\calE_{\distinct}^{\leq r}$ and $\frakP\in\mspec\calO_{K}$,
we shall define a tree $T_{G}^{\frakP}$,
and one variety $V_{y}^{\frakP}$ for each vertex $y$ of $T_{G}^{\frakP}$.
The construction of the tree $T_{G}^{\frakP}$
and the varieties $V_{y}^{\frakP}$ is similar to the construction of
$T_{G}^{\infty}$ and $V_{x}^{\infty}$ above, with the following
changes: For the root $x_{0}^{\frakP}$ of $T_{G}^{\frakP}$, we set $V_{x_{0}^{\frakP}}^{\frakP}$
equal to $\AA_{(\calO_K/\frakP)^{\al}}^{r}$ rather than $\AA_{\QQ^{\al}}^{r}$,
and use the polynomials $\left(g_{1}\right)_{\frakP},\dots,\left(g_{s}\right)_{\frakP}$
instead of $g_{1},\dotsc,g_{s}$ (where $g_{\frakP}$ is the reduction
mod $\frakP$ of $g\in\calO_K\left[T_{1},\dotsc,T_{r}\right]$).

Next, for each vertex $x$ of $T_{G}^{\infty}$, we shall define
an ideal $J_{x}$
of $\QQ^{\al}[T_1,\dotsc,T_r]$
such that
$V(J_{x})=V_{x}^{\infty}$.
We set
$J_{x_{0}}=(0)$
for the root $x_{0}$ of $T_{G}^{\infty}$.
Recall that the rooted tree $T_{G}^{\infty}$ has $s+1$ levels, and
$s\leq r$. Thus, by B\'{e}zout's inequality 
 $T_{G}^{\infty}$ has at most $d_{0}^{r}$ leaves and no more than $\max\left\{ r,2d_{0}^{r}\right\}$
vertices (considering separately the cases $d_0=1$ and $d_0 \geq 2$). By repeated application of Theorem \ref{thm:good-reductions},
we obtain an ideal
$J_x$ for each vertex $x$ of $T_{G}^{\infty}$, such that:\renewcommand{\labelenumi}{\Roman{enumi})}
\begin{enumerate}
\item For every vertex $x$ of $T_{G}^{\infty}$:
\begin{enumerate}
\item $V(J_x)=
V_{x}^{\infty}$ and $F_{J_{x}}=F_{V_{x}^{\infty}}$.
\item $J_{x}$ has complexity at most $(d_{0}^{O_{r}(1)},
O_r(d_{0}^{O_{r}(1)}h_{0}))$.
\end{enumerate}
\item There is $\Delta_{1,G}\in\NN$, $\log\Delta_{1,G}\ll_r [F:\QQ]^{O(1)}d_{0}^{O_{r}\left(1\right)}h_{0}$,
such that for every
$\frakP\in\mspec\calO_{K}$
such that
$\gcd(\calN\frakP,\Delta_{1,G})=1$,
there is a bijection $\varphi_{G}^{\frakP}$ from the set of vertices of $T_{G}^{\infty}$
to the set of vertices of $T_{G}^{\frakP}$, such that:
\begin{enumerate}
\item $\varphi_{G}^{\frakP}$ induces an isomorphism of rooted trees.
\item
$V((J_{x})_\frakP)=
V_{\varphi_{G}^{\frakP}\left(x\right)}^{\frakP}$
for each vertex $x$ of $T_{G}^{\infty}$.
\end{enumerate}
\end{enumerate}

\textit{The definitions of $\Delta_{1}$ and $\calB$:} Set $\Delta_{1}=\prod_{G}\Delta_{1,G}$,
where $G$ runs over $\calE_{\distinct}^{\leq r}$. Then $\log\Delta_{1}\ll
|\calE|^{r}[F:\QQ]^{O(1)} d_{0}^{O_{r}\left(1\right)}h_{0}$.
Now, take $\frakP\in\mspec\calO_{K}$
such that $\gcd(\calN\frakP,\Delta_{1})=1$, let
$E\subseteq\calE$, and take an irreducible component $C$
of $V\left(E_{\frakP}\right)$. By Lemma \ref{lem:component-of-few-eqs-1},
there is $G\in\calE_{\distinct}^{\leq r}$, such that $C=V_{y}^{\frakP}$
for some leaf $y$ of $T_{G}^{\frakP}$. By II(b), and since $\gcd(\calN\frakP,\Delta_{1,G})=1$,
there is a leaf $x$ of $T_{G}^{\infty}$ such that $V_{y}^{\frakP}=
V((J_{x})_\frakP)$.
Let 
\[
\calB=\left\{ J_{x}\mid G\in\calE_{\distinct}^{\leq r}\text{, \ensuremath{x} is a leaf of \ensuremath{T_{G}^{\infty}}}\right\} \,\,\text{.}
\]
Then (1) holds. Furthermore 
\begin{equation}
|\calB|\leq|\calE|^{r}d_{0}^{r}\label{eq:types-proof-bound-on-B}
\end{equation}
since each tree $T_{G}^{\infty}$ has at most $d_{0}^{r}$ leaves.

\textit{The definition of $\Delta_{2}$:} Take $J\in\calB$ and $f\in\calE$. By Proposition
\ref{prop:inclusion-preserved-mod-p}, there is $\Delta_{2,J,f}\in\NN$,
$\log\Delta_{2,J,f}\ll_r [F_{J}F:\QQ]^{O(1)}d_{0}^{O_{r}\left(1\right)}h_{0}\ll_r 
[F:\QQ]^{O(1)}d_{0}^{O_{r}\left(1\right)}h_{0}$, such that
for every $\frakP\in\mspec\calO_{K}$  with $\gcd(\calN\frakP,\Delta_{2,J,f})=1$, $V\left(J_{\frakP}\right)\nsubseteq V\left(\{f\}_{\frakP}\right)$ if and only if $V(J)\nsubseteq V(\{f\}).$ 
Set $\Delta_{2}=\prod_{J,f}\Delta_{2,J,f}$,
where $J$ runs over $\calB$ and $f$ runs over $\calE$. Then $\log\Delta_{2}\ll_r |\calE|^{r+1}[F:\QQ]^{O(1)} d_{0}^{O_{r}\left(1\right)}h_{0}$
by \eqref{eq:types-proof-bound-on-B}.
Take $\frakP\in\mspec\calO_{K}$  with $\gcd(\calN\frakP,\Delta_{2})=1$, $E\subseteq\calE$
and $J\in\calB$. We thus have  $V(J_{\frakP})\subseteq V\left(E_{\frakP}\right)$ if and only if $V(J) \subseteq V(E)$. Thus (2) holds. 

\end{proof}

Finally, we record the following lemma. It gives a bound on the degree of subvarieties defined by the vanishing of words of length $\ell$. For a word $w$ on $k$ letters and their inverses, say $\{ s_{1},\dotsc,s_{k},s_{1}^{-1},\dotsc,s_{k}^{-1}\}$,  and elements $g_{1},\dotsc,g_{k}$ of a group $G$, write $w(g_{1},\dotsc,g_{k})$
for the element of $G$ resulting from the substitution $s_{i}\leftarrow g_{i}$.
Recall that $F_0$ is either $\QQ$ or $\FF_p(t)$ and $F_0^{\al}$ is an algebraic closure of $F_0$.

\begin{lem}
\label{lem:deg-height--word-eqs}For such a word $w$ of length
$\ell$ let
\[
Y_{w}=\left\{ \left(g_{1},\dotsc,g_{k}\right)\in \SL_{n}(F_0^{\al})^{k}\mid w\left(g_{1},\dotsc,g_{k}\right)=1\right\} \,\,\text{.}
\]
We view $\SL_{n}$ as a closed subvariety of $M_{n\times n}$, and identify the latter with the affine space $\AA_{F_0^{\al}}^{n^{2}}$. Then
$Y_{w}=\SL_{n}(F_0^{\al})^{k}\cap V(I_{F_0^{\al}})$
where $I$ is an ideal of $R:=\ZZ[\{ T_{t,i,j}\} _{i,j\in[n], t\in[k]}]$
generated by $n^{2}$ integer polynomials $f_{1},\dotsc,f_{n^{2}}$ of degree  at most $n\ell$ and whose coefficients have
absolute value at most $\left(n!\right)^\ell$, and $I_{F_0^{\al}}=I \otimes_\ZZ F_0^{\al}$.
\end{lem}

\begin{proof}
For a word $w$ and $A_{1},\dotsc,A_{k}\in M_{n\times n}\left(R\right)$,
write $w\left[A_{1},\dotsc,A_{k}\right]$ for the element of $M_{n\times n}\left(R\right)$
resulting from applying the substitution $s_{i}\leftarrow A_{i}$, $s_{i}^{-1}\leftarrow\adj A_{i}$
to $w$ (where $\adj A=\left(\left(-1\right)^{i+j}M_{ji}\right)_{ij}\in M_{n\times n}\left(R\right)$,
where $M_{ji}$ is the $\left(j,i\right)$-minor of $A$). Let $X_{t}=\left(T_{t,i,j}\right)_{i,j\in\left[n\right]}\in M_{n\times n}\left(R\right)$
for each $t\in\left[k\right]$.

For a multivariate complex polynomial $f$, write $\supp f$
for the multiset of nonzero coefficients of $f$ and 
$\ell_{\infty}(f)$ for  the absolute value of the largest
coefficient. For a matrix $A\in M_{n\times n}\left(R\right)$,
let $\ell_{\infty}\left(A\right)=\max_{i,j}\ell_{\infty}\left(A_{ij}\right)$
and $\deg A=\max_{i,j\in\left[n\right]}\deg A_{ij}$.

Clearly, $\deg\left(\adj A\right)\leq n\deg\left(A\right)$, and $\deg\left(AB\right)\leq\deg\left(A\right)+\deg\left(B\right)$
for $A,B\in M_{n\times n}\left(R\right)$. Thus, for a word $w$
of length $\ell$ and $A_{1},\dotsc,A_{k}\in M_{n\times n}\left(R\right)$,
we have $\deg\left(w\left[A_{1},\dotsc,A_{k}\right]\right)\leq n\ell\max\left\{ \deg\left(A_{1}\right),\dotsc,\deg\left(A_{k}\right)\right\} $.
So
\[
\deg\left(w\left[X_{1},\dotsc,X_{k}\right]\right)\leq n\ell
\]
because $\deg X_{t}=1$ for each $t\in\left[k\right]$.

Clearly, $\ell_{\infty}\left(fg\right)\leq\min\left\{ |\supp f|,|\supp g|\right\} \ell_{\infty}\left(f\right)\ell_{\infty}\left(g\right)$
for all $f,g\in R$. Thus, for $A,B\in M_{n\times n}\left(R\right)$
and $i,j\in\left[n\right]$, denoting 
\[
m_{0}=\min\left\{ \max_{i,j}|\supp A_{ij}|,\max_{i,j}|\supp B_{ij}|\right\} \,\,\text{,}
\]
we have
\begin{align*}
\ell_{\infty}\left(\left(AB\right)_{ij}\right) & \leq\sum_{l=1}^{n}\ell_{\infty}\left(A_{il}B_{lj}\right)\leq nm_{0}\ell_{\infty}\left(A\right)\ell_{\infty}\left(B\right)\,\,\text{.}
\end{align*}
Furthermore, $\ell_{\infty}(X_{t})=1$, $\ell_{\infty}(\adj X_{t})=1$,
$\max_{i,j}|\supp\left(X_{t}\right)_{ij}|=1$ and $\max_{i,j}|\supp(\adj X_{t})_{ij}|=(n-1)!$
for each $t\in\left[k\right]$. Thus,
\begin{equation*}
\ell_{\infty}(w[X_{1},\dotsc,X_{k}]) \leq (n!)^{\ell-1}.
\end{equation*}
\end{proof}

\section{Uniform expansion at almost all primes}\label{sec-prime}

Let $k$ be an algebraically closed field and $\GG$ a connected semisimple algebraic group defined over $k$. We fix an embedding $\GG \hookrightarrow \GL_d$ so we can talk about the degree of closed subvarieties of $\GG$. Let $M\ge 1$ be a parameter. Borrowing terminology from \cite{breuillard-green-tao-linear} (see also \cite{breuillard-standrews}) we will say that a subset $\Omega$ of $\GG(k)$ is $M$-dense if $\Omega$ is not contained in a proper algebraic subvariety of $\GG$ with degree at most $M$. 
Let $\FF_p^{\al}$ be the algebraic closure of $\FF_p$. 

It is known from the work of Larsen and Pink \cite{larsen-pink}, which extends Nori's work \cite{nori},  that when $\GG$ is absolutely almost simple of adjoint type, there is $M_0$ depending only on $\dim \GG$ such that every finite subgroup $G$ of $\GG(k)$ that is $M_0$-dense is a \emph{subfield subgroup} of $\GG$.  This means that $k$ has positive characteristic $p$ and that there is a Frobenius map $F$ of $\GG(k)$ (the composition of a field automorphism and a Dynkin diagram automorphism) such that $G$ is contained in the subgroup of fixed points $H:=\GG(k)^F$ and contains the commutator subgroup $[H,H]$, which has bounded index in $H$ (a uniform bound depending on $\dim \GG$ only). Moreover, apart from a finite number of groups (see \cite[Theorem 2.2.7]{gorenstein-lyons-solomon3}) $[H,H]$ is a finite simple group.  Conversely any such group is $M_0$-dense and every finite simple group of Lie type arises this way. We refer the reader to \cite{breuillard-green-tao-linear}, \cite{breuillard-standrews} and \cite[Section 5]{bggt} for further discussion of these facts in the context of expanders. 

Theorem \ref{cro-refor} from the introduction is thus a special case of the following.

\begin{thm}(uniform expansion at almost all primes)\label{expmain} For every  $r,d\ge 1$ and $\delta>0$, there is $M_0(d,r) \in \NN$ and $\eps=\eps(d,r,\delta)>0$ and a (possibly empty) set of prime numbers $\mathcal{B}_\delta$ with
\begin{equation}
\left|\left\{ p\in\calB_{\delta}\mid p<X\right\} \right|\leq X^{\delta}\qquad\forall X\gg_\delta 1 \label{eq:gaps-bad-primes-sparse}
\end{equation}
such that if  $\GG_\ZZ$ is a semisimple Chevalley $\ZZ$-scheme of $\GL_d$, $q=p^s$ with $s\leq r$ and $p\notin \mathcal{B}_\delta$, and $G\leq \GG_p(\FF_q)$  is an $M_0$-dense subgroup, then $G$ is $\eps$-expanding, i.e. \eqref{expand} holds for all $A,B \subset G$, $|B|\leq |G|/2$, $A \ni 1$ generating. 
\end{thm}

Here $\GG_p$ denotes the algebraic group $\GG_\ZZ \otimes \FF_p$.  Before we delve into the proof of this result, we derive some consequences.

\begin{cor}\label{expansion} Given $n,\delta>0$, there is $\eps=\eps(n,\delta)>0$ and a set of prime numbers $\mathcal{B}_\delta$ satisfying \eqref{eq:gaps-bad-primes-sparse} such that for every prime $p\notin \mathcal{B}_\delta$ and every semisimple algebraic group $\GG$ defined over $\FF_p$ and of (absolute) rank at most $n$, $\GG(\FF_p)$ is $\eps$-expanding.
\end{cor}


\begin{proof}


The group $\GG$ splits over $\FF_p^{\al}$, so is an $\FF_p$-form of a semisimple Chevalley $\ZZ$-scheme $\GG_\ZZ$ of rank $\leq n$ and $\GG(\FF_p)\leq \GG_p(\FF_p^{\al})$. There are only boundedly many (in terms of $n$) such forms, as they are in correspondence with $1$-cocycles from  $\Gal(\FF_p^{\al}/\FF_p)$ to the group of outer automorphisms of $\GG_\ZZ$ (see \cite[Theorem 2 and 3]{hertzig}). Consequently $\GG$ splits over $\FF_q$ for some $q \leq p^r$, $r=r(n)$ and so $\GG(\FF_p)\leq \GG_p(\FF_q)$.  We are thus in a position to apply Theorem \ref{expmain} and conclude. It only remains to check that
for a given $M>0$, $\GG(\FF_p)$ is $M$-dense in $\GG_p(\FF_p^{\al})$ as soon as $p$ is large enough in terms of $M$. This follows from the Lang--Weil estimate and the Schwartz-Zippel bound (see \cite[Section 6]{bggt}) since $\GG(\FF_p)$ has size $p^d +O_d(p^{d-1/2})$ where $d=\dim \GG$.
 \end{proof}

Similarly, we derive the following, which will be needed in our forthcoming paper  \cite{becker-breuillard-varju}.

\begin{cor}\label{nfields}Let $K$ be a number field and $\GG$ a semisimple algebraic $K$-group. Let $\delta>0$. There is $\eps>0$ and a subset $\mathcal{B}_\delta$ of rational primes satisfying \eqref{eq:gaps-bad-primes-sparse} such that if $\mathfrak{p}$ is a prime ideal of $\mathcal{O}_K$ with residue field of characteristic $p\notin \mathcal{B}_\delta$, then $\GG(\mathcal{O}_K/\mathfrak{p})$ is $\eps$-expanding. Moreover $\eps$ is bounded below in terms of $\delta$, $[K:\QQ]$ and $\dim \GG$ only.
\end{cor}

\begin{proof}Let $\GG'$ be the restriction of scalars of $\GG$ from $K$ to $\QQ$ (see \cite[2.1.2]{platonov-rapinchuk}). Then when $p$ is large enough, $\GG'_p(\FF_p) \simeq \prod_{\mathfrak{p}|p} \GG_{\mathfrak{p}}(\mathcal{O}_K/\mathfrak{p})$. So the $\GG_\mathfrak{p}(\mathcal{O}_K/\mathfrak{p})$ appear as quotients of $\GG'_p(\FF_p)$, so it is enough to prove the expander property for $\GG'_p(\FF_p)$ (see Lemma \ref{quotoverspec}). This is the group of $\FF_p$-points of a semisimple algebraic group defined over $\FF_p$ and of dimension bounded in terms of $\dim \GG$ and $[K:\QQ]$ only. The result then follows from Corollary \ref{expansion}.
\end{proof}

\subsection{Non-degenerate locus and generating tuples}

\begin{prop}\label{zdense} Let $\GG$ be a connected semisimple algebraic group over a field $K$ of characteristic zero and $m\ge 2$. For an $m$-tuple in $\GG^m$, generating a Zariski-dense subgroup is a Zariski-open condition (defined over $K$). Moreover, there is $M_1=M_1(\dim \GG) \ge 1$ such that every $M_1$-dense subgroup is Zariski-dense. 
\end{prop}

\begin{proof} This is well-known. See for instance  \cite[Prop. 8.2]{acampo-burger} and \cite{guralnick}. More precisely, there are two irreducible $K$-representations $\rho_1$ and $\rho_2$ of $\GG$ such that for any tuple $(g_1,\ldots,g_m) \in \GG$, the subgroup $\Gamma$ it generates is Zariski-dense if and only if $\Gamma$ acts irreducibly via $\rho_1$ and $\rho_2$. Since acting irreducibly is a Zariski-open condition, the conclusion follows. The last assertion is proven in \cite[Lemma 7.4]{becker-breuillard}.
\end{proof}

In positive characteristic, the (finite) subfield subgroups are not Zariski-dense, but they can be $M$-dense for large $M$. Nevertheless, the condition of generating a subgroup that is not $M$-dense (for a given $M$) is a Zariski-closed condition. In fact:

\begin{prop}\label{degprop}Let $\GG$ be a semisimple Chevalley group scheme over $\ZZ$ and $M \ge 1$. Given $m\geq2$, there is a proper closed $\ZZ$-subscheme $V_{M}^{\left(m\right)}$ of $\GG_{\ZZ}^{m}$, which over $k=\QQ^{\al}$ (resp.  $k=\FF_p^{\al}$) is the algebraic set of $m$-tuples in $\GG^m$ (resp. $\GG_p^m$) that generate a non $M$-dense subgroup.
\end{prop}

\begin{proof} We view $\GG \leq \GL_{n-1}$ as a subscheme of the affine space of square matrices $M_{n}$. By escape from subvarieties (see \cite[Lemma 3.1]{becker-breuillard}) an $m$-tuple $(g_1,\ldots,g_m)$ does not generate an $M$-dense subgroup if and only if the words of length at most $C=C(n,M)$ in $(g_1,\ldots,g_m)$ all lie in some proper closed subvariety (say a hypersurface) of degree at most $M$. This condition can be expressed in terms of the polynomials defining the hypersurfaces, looking at their coefficients. Polynomials of degree at most $M$ in the matrix coefficients form a linear space of dimension $M^{O_{m,n}(1)}$. To say that one of them vanishes on these words without vanishing identically on $\GG^m$ exactly means that the linear forms on the space of polynomials obtained by evaluating at each of these words do not span the space of all such linear forms. This can be expressed by the vanishing of the minors of maximal dimension. Hence by the vanishing of polynomials (with integer coefficients) of degree at most $M^{O_{m,n}(1)}$ in the matrix entries of $(g_1,\ldots,g_m)$. These define the subscheme $V_{M}^{\left(m\right)}$.
\end{proof}

We  denote by $V_{M,p}^{(m)}$ the variety over $\FF_{p}^{\al}$ resulting from $V_{M}^{\left(m\right)}$ by base change. We show:

\begin{lem}\label{andrewslem} Given $M\ge M_1$, there is $p_M$ depending only on $M, m,\dim \GG$ such that if $p>p_M$ then $V^{(m)}_{M,p}=V^{(m)}_{M_1, p}$.
\end{lem}

\begin{proof}This is \cite[Lemma 2.7]{breuillard-standrews}. The proof given there relied on ultraproducts. Instead, one can argue directly using Proposition \ref{prop:inclusion-preserved-mod-p}. 
Since $M\ge M_1$, Proposition \ref{zdense} implies that $V_{M}^{\left(m\right)}$ is the locus of $m$-tuples that do not generate a Zariski-dense subgroup of $\GG$. By Proposition \ref{prop:inclusion-preserved-mod-p},  we get that $V_{M,p}^{\left(m\right)}$
and $V_{M_1,p}^{\left(m\right)}$ coincide if $p\gg_{M,n} 1$. (This approach has the merit of giving an explicit bound on $p_M$; tracing the height of the polynomials that define $V_{M}^{(m)}$ one sees that $\log p_M \ll_n M^{O_{n,m}(1)}$.)
\end{proof}

\subsection{Subgroup structure}
We record here a preliminary analysis of the subgroups of $G$, in the form of the following lemma, which exploits Nori's theorems from \cite{nori}.

\begin{lem}\label{norilem}
Let $G$ be a subgroup of $\GG_{p}\left(\FF_{q}\right)$, where $\GG_{\ZZ}$
is a semisimple Chevalley $\ZZ$-scheme of $\GL_{d}$, $q=p^{r}$,
$p$ a prime number. Let $G^{+}$ be the subgroup of $G$ generated
by elements of order $p$. Then there is a connected semisimple algebraic group $\widetilde{G}$ defined over $\FF_p$ such that $G^+= \widetilde{G}(\FF_p)^+$. Moreover, $[G:G^+]$ is bounded in terms of $d$ and $r$ only. Furthermore, there is $M_0=M_0(d,r)\ge 1$, such that if $G$ is $M$-dense in $\GG_p$ and $M\ge M_0$,  $\widetilde{G}$ is isomorphic over $\FF_p^{\al}$ to a cartesian power of at most $r$ copies of $\GG$,  and if $p\gg_{M,r,d} 1$, $G^+$ is the only subgroup of $G^+$ that is $M$-dense in $\widetilde{G}$.
\end{lem}

\begin{proof}
Consider the restriction of scalars $\HH_{p}\coloneqq\Res_{\FF_{p}}^{\FF_{q}}\GG_{p}$.
This is a connected algebraic group defined over $\FF_{p}$, which is an $\FF_{p}$-form of $\GG_{p}^{r}$. We have an isomorphism
of algebraic groups $\HH_{p}\cong\GG_{p}^{r}$ over $\FF_{p}^{\al}$
(in fact, over $\FF_{q}$) of constant degree (in fact, induced by
a linear map), and an isomorphism of abstract groups $\GG_{p}\left(\FF_{q}\right)\to\HH_{p}\left(\FF_{p}\right)$,
such that $g\in\GG_{p}\left(\FF_{q}\right)$ is mapped to
$(g,g^{\sigma},\dotsc,g^{\sigma^{r-1}})\in\GG_{p}^{r}\left(\FF_{q}\right)$,
where $\sigma$ is the Frobenius map.
Hence we may view $G$ as a subgroup of $\HH_p(\FF_p)$. By Nori's theorem \cite[Theorem B]{nori}, if $p$ is larger than a constant depending only on $r$ and $d$, there is a connected algebraic subgroup $\widetilde{G}$ of $\HH_p$ defined over $\FF_p$ such that $\widetilde{G}(\FF_p)^+=G^+$, where the superscript $+$ denotes  the subgroup generated by elements of order $p$. Note that $G^+$ is normal in $G$. Moreover the index $[\widetilde{G}(\FF_p):\widetilde{G}(\FF_p)^+]$ is bounded by a constant depending  only on $d,r$  by \cite[Theorem C and Remark 3.6]{nori}. Also Nori's group $\widetilde{G}$ is generated by unipotent subgroups. In particular, its degree is bounded in terms of $d$ and $r$ only.

By \cite[Theorem C]{nori} $G$ has a commutative subgroup $H$ of order prime to $p$ such that $HG^+$ is normal in $G$ and has index bounded in terms of $d,r$ only. The projection of $G$ to each of the $r$ coordinate factors of  $\GG^r_p$ is sufficiently Zariski-dense there by assumption. It follows that the same holds for $HG^+$ and hence also for its derived group 
(indeed, if a subgroup $L$ is sufficiently Zariski dense in an algebraic group $\KK$, then considering the multiplication map $\KK\times\KK\to\KK$, we see that $L\times L$ is sufficiently Zariski dense in $\KK\times\KK$, and thus so is $[L,L]$ in $\KK$),
which is contained in $G^+$.
Hence the same holds for $G^{+}$, and thus for $\widetilde{G}$. That
is, we have shown that the projection of $\widetilde{G}$ to each factor
of $\GG_{p}^{r}$ is $M$-Zariski dense for any given $M>M_0$ as long
as $M_{0}$ is chosen large enough (as a function of $d$ and $r$
only). On the other hand, each projection of $\widetilde{G}$ is closed
and of degree bounded in terms of $d$ and $r$ only, since the same
is true of $\widetilde{G}$.
For this we needed to assume $M_0$ sufficiently large depending on $d$ and $r$ only. As a result, the algebraic subgroup $\widetilde{G}$, which is of bounded degree, must project onto each factor of $\GG^r_p$. In other words, it is a connected subdirect product of $\GG^r_p$.
So, the centralizer of $\widetilde{G}$ in $\HH_p$ is contained in the center $Z$ of $\HH_p$. 
Moreover $\widetilde{G}$ is semisimple and isomorphic (over $\FF_p^{\al}$) to
$\GG^{r'}_p$ for some $r'\leq r$.
Indeed, if for some $i\neq j$, $\tilde{G}$ does not
surject under $\pi_{i}\times\pi_{j}\colon\tilde{G}\to\GG_{p}\times\GG_{p}$
then $\left(\pi_{i}\times\pi_{j}\right)(\tilde{G})\cong\GG_{p}$
by Goursat's lemma and since $\tilde{G}$ is connected, and we can
continue by induction. If no such pair $i\neq j$ exists, then, as
$\GG_{p}$ is perfect, $\tilde{G}=\GG_{p}^{r}$ (see e.g.
\cite[Lemma 3.3]{RibetModularForms}).

Note that $H$ normalizes $G^+$, and hence normalizes $\widetilde{G}$. Indeed, if $g$ normalizes $G^+$ but not $\widetilde{G}$, then $g \widetilde{G} g^{-1} \cap \widetilde{G}$ is a proper algebraic subgroup of $\widetilde{G}$, and hence (by Lang--Weil or Schwartz--Zippel, see \cite[Section 6]{bggt}) the number of its $\FF_p$-points is much smaller than $|G^+|$, which has bounded index in $\widetilde{G}(\FF_p)$
in terms of $d$ and $r$ only. This is
a contradiction as long as $p$ is large enough in terms of $d$ and $r$, which we are allowed to assume.
Now, $\Out(\widetilde{G})$ is bounded in terms of $\dim \widetilde{G}$ only, as $\widetilde{G}$ is semisimple.
Consequently, the kernel $H_{0}$
of the homomorphism $H\to\Out(\tilde{G})$ has bounded
index in $H$. For every $h_{0}\in H_{0}$ there is $g_{0}\in\tilde{G}(\FF_{p}^{\al})$
such that $g_{0}^{-1}h_{0}$ centralizes $\tilde{G}$ and thus $g_{0}^{-1}h_{0}\in Z$.
Consider the map $i_Z:H\to  H$ sending $h$ to $h^{|Z|}$. We conclude that $i_Z(H_0)\leq \widetilde{G}(\FF_p^{\al})$.
But $H\leq \HH_p(\FF_p)$, so in fact  $i_Z(H_0)\leq \widetilde{G}(\FF_p)$. In particular $[H_0:H_0\cap \widetilde{G}(\FF_p)]\leq [H_0:i_Z(H_0)]$. But since $H\leq \GG_p(\FF_q)$ is abelian of order prime to $p$, it belongs to a maximal torus defined over $\FF_q$. This torus must split over a finite extension of $\FF_q$. Hence $H$ is isomorphic to a subgroup of $(\FF_{q'}^{\times})^{m}$, where $m$ is the absolute rank of $\GG_p$ and $q'$ a power of $q$. This is an abelian group on at most $m$ generators. It follows that $[H_0:i_Z(H_0)] \leq  |Z|^{m}$. We conclude that $[H:H\cap \widetilde{G}(\FF_p)]\leq [H:H_0]|Z|^m$ is bounded in terms of $d$ and $r$ only. On the other hand, $\widetilde{G}(\FF_p)^+=G^+$ and  $[\widetilde{G}(\FF_p): \widetilde{G}(\FF_p)^+]$ is similarly bounded, hence so are $[H:H\cap G^+]$, $[HG^+:G^+]$ and $[G:G^+]$.


We end this discussion by establishing the last clause of the lemma. First, note that, since the index of $G^+=\widetilde{G}(\FF_p)^+$ in $\widetilde{G}(\FF_p)$ is bounded in terms of $r$ and $d$ only, and since $\widetilde{G}$ is connected,  the Lang--Weil estimates imply that $G^+$ must be $M$-dense in $\widetilde{G}$ as soon as $p$ is large enough, say $p\ge p_{M,r,d}$. Now suppose that $M \ge M_0$ and that the subgroup $L\leq G^+$ is $M$-dense in $\widetilde{G}$. Then its projection to each $\GG_p$ factor is $M_0$-dense. So $L$ satisfies the same hypotheses as $G$ and we conclude by the above that there is a connected semisimple algebraic subgroup $\widetilde{L}$ of $\HH_p$ defined over $\FF_p$ such that $L \cap \widetilde{L}(\FF_p)$ has bounded index in $L$, say $[L:L\cap \widetilde{L}(\FF_p)]=O_{r,d}(1)$ and $L^+=\widetilde{L}(\FF_p)^+$. It follows that $L\cap \widetilde{L}(\FF_p)$, and hence $\widetilde{G}\cap \widetilde{L}$ is $M/O_{r,d}(1)$-dense in $\widetilde{G}$ and hence $\widetilde{L} = \widetilde{G}$. So $L^+=G^+$ and thus $L=G^+$.
\end{proof}

\noindent \emph{Example.} Suppose $p>3$ and $G$ is the unitary group $\SU_n(p):=\{g \in \SL_n(\FF_{p^2}) | g^Tg^\sigma = 1\}$ where $\sigma:x \mapsto x^p$ is the Frobenius automorphism. Then $G=G^+$, $\GG=\SL_n$,  $q=p^2$ and $r=2$, $\widetilde{G}$ is a non-split $\FF_p$-form for $\SL_n$ and a proper $\FF_p$-subgroup of $\Res^{\FF_q}_{\FF_p} \GG_p$. On the other hand, if $G$ lies in between $\PSU_n(p)$ and $\PGU_n(p):=\{g \in \PGL_n(\FF_{p^2})| g^Tg^\sigma = 1\}$, then $G^+=[G,G]=\PSU_n(p)$,  $[G:G^+]\leq [\PGU_n(p):\PSU_n(p)]=\gcd\{n,q+1\}$, $\GG=\PGL_n$, $r=2$, and $\widetilde{G}(\FF_p)=\PGU_n(p)$. See \cite[chapter 11]{grove} for a detailed study of unitary groups.
\bigskip


From Proposition \ref{degprop} and Lemma \ref{andrewslem} we conclude:

\begin{lem}\label{generating} In the setting of Lemma \ref{norilem}, there is $p_0=p_0(d, m, r)>1$ such that if $p>p_0$, then $(g_1,\ldots,g_m) \in G^+$ generates $G^+$ if and only if $(g_1,\ldots,g_m) \notin V^{(m)}_{M_1,p}(\widetilde{G})$.
\end{lem}

\subsection{Bourgain-Gamburd method}
Recall the ingredients of the Bourgain-Gamburd method (\cite{bourgain-gamburd}, \cite[Proposition 3.1]{bggt}, \cite[Proposition 3.1]{breuillard-standrews}, \cite{tao-expansion-book}). For a finite group $G$ with symmetric generating set $S$ assume: a) \emph{quasirandomness}, i.e. the property that the dimension of any non-trivial representation is at least $|G|^{\kappa_1}$, b) a \emph{product theorem} namely the fact that $|AAA| \ge \max\{|G|,|A|^{1+\eps_0}\}$ for every generating set $A \subset G$, and c) an exponential \emph{non-concentration estimate} on proper subgroups asking for some $\ell\geq \kappa_2 \log |G|$ such that
\begin{equation}\label{noncon}\mu_S^\ell(H) \leq \exp(-c\ell)\end{equation}
for every proper subgroup $H$ of $G$. Here $\mu_S:=\frac{1}{2|S|}\sum_{s \in S} (\delta_s+\delta_{s^{-1}})$. The conclusion is that $\Cay(G,S)$ is an $\eps$-expander graph, namely \eqref{expand} holds for $A=S$ and any $B$ with $|B|\leq |G|/2$,  for some $\eps=\eps(\kappa_1,\eps_0,c,\kappa_2)>0$. See \cite[Proposition 3.1]{breuillard-standrews} for quantitative information on how the lower bound $\eps$ depends on the parameters.


Quasirandomness of the groups $G$ appearing in Theorem \ref{expmain} follows from the quasi-randomness of finite simple groups of Lie type \cite{landazuri-seitz}. Indeed, the group $\widetilde{G}(\FF_p)$ from Lemma \ref{norilem} is the group of $\FF_p$-points of a connected semisimple algebraic group defined over $\FF_p$ and $\widetilde{G}(\FF_p)^+$ is isogenous to a product of quasisimple groups (see e.g. \cite[Theorem 24.17]{malle-testerman}). The exponent $\kappa_1$ will depend only on $r=\log q/\log p$ and $\dim \GG$. The product theorem for $G$ was proved in \cite{breuillard-green-tao-linear} and \cite{pyber-szabo} when $\GG$ is absolutely simple. The semisimple case is treated in \cite[Theorem 8.1]{bggt} as well as in \cite[Theorem 4.13]{breuillard-standrews}  (see also \cite{eberhard-pyber-szabo} for the most general version of the product theorem in the non-semisimple case). Therefore, the only condition left to be checked is the non-concentration estimate on proper subgroups \eqref{noncon}. 



\subsection{Proof of Theorem \ref{expmain}}
\label{subsec-proofmain}

Lemma \ref{norilem} tells us that $G^+$ has bounded index in $G$. Using the following lemma we can thus assume without loss of generality that 
$G=G^+=\widetilde{G}(\FF_{p})^{+}$.

\begin{lem}\label{quotoverspec} If a finite group $G$ is $\eps$-expanding,  then so is every quotient $G/N$, $N \vartriangleleft G$. Conversely, if $H\leq G$ is an $\eps$-expanding subgroup of index $m$, then $G$ is $\eps/2m$-expanding.
\end{lem}

\begin{proof}The first assertion is immediate from \eqref{expand}.  For the second assertion one decomposes $B \subset G$ into $H$-cosets and exploits the fact that for every generating set of $H$ there are words of length at most $2m$ that form a generating set of $G$ (e.g. \cite[C.1]{breuillard-green-tao-linear}). 
\end{proof}

Note that in order to prove the theorem, it is enough to assume that the generating set $A$ consists of two elements, their inverses and the identity.  This follows from Lemma \ref{generating} and from escape from subvarieties (see \cite[Lemma 3.1]{becker-breuillard}) applied to the field $K=\FF_p^{\al}$, the set
$\Sigma= (A \times \{1\}) \cup  (\{1\} \times A) \cup \{(1,1)\}$ in $\GG_{p} \times \GG_{p} \leq \GL_d \times \GL_d \leq \GL_{2d}$
and the variety $V_{M_1,p}^{(2)}$. Indeed, this gives a constant $N=N(d)$ such that for every generating set $A$ of $G^+$, one can always find a pair of generators expressible as words of length at most $N$ with letters in $A$. If this pair is $\eps$-expanding, then $A$ is $\eps/N$-expanding.

The algebraic group $\widetilde{G}$ obtained in Lemma \ref{norilem} is an $\FF_p$-form of a cartesian power $\GG_p^{r_p}$, for some $r_p\leq r$. Below we denote by $\HH_\ZZ$ any cartesian power $(\GG_\ZZ)^i$, $i\in \{1,\ldots,r\}$. We view $\HH_\ZZ$ as a $\ZZ$-subscheme of $\SL_{n}$ for some fixed $n=n(d,r)$. We let $\HH=\HH_\ZZ\otimes \QQ^{\al}$ and $\HH_p=\HH_\ZZ \otimes \FF_p$.



Let $W_{\ell}$ be the set of all pairs $\left(w_{1},w_{2}\right)$, where each $w_{i}$
is a word (not necessarily reduced) of length $\ell$ over the alphabet
$\{ x_{1}^{\pm 1},x_{2}^{\pm 1}\} $.
For $\overline{w}=\left(w_{1},w_{2}\right)\in W_{\ell}$
and $\overline{g}=\left(g_{1},g_{2}\right)$,
where $g_1$ and $g_2$ are elements of some group $L$, let
$\overline{w}\left(\overline{g}\right)=\left(w_{1}\left(\overline{g}\right),w_{2}\left(\overline{g}\right)\right)$,
where $w_{i}\left(\overline{g}\right)$ is the element of $L$ resulting
from applying the substitution $x_{i}\leftarrow g_{i}$ to $w_{i}$.
By Theorem 1.3 in \cite{becker-breuillard} applied to $\HH\times \HH$ and the uniform measure on the $16$ pairs $(g_i^{\pm1},g_j^{\pm1})$, there is $C=C\left(d\right)$ such
that
\[
\PP\left(\overline{w}\left(\overline{g}\right)\in V_{M_1}^{\left(2\right)} \right)\leq Ce^{-\ell/C}
\]
for every $\ell\geq1$ and $\overline{g}\in \HH^{2}\setminus V_{M_1}^{\left(2\right)}$, where $\overline{w}$ is chosen uniformly at random  in $W_\ell$.
Equivalently, for every subset $W'_{\ell}$ of $W_{\ell}$, if $\left|W'_{\ell}\right|>Ce^{-\ell/C}\cdot\left|W_{\ell}\right|$
then
\[
\left\{ \overline{g}\in \HH^{2}\mid\overline{w}\left(\overline{g}\right)\in V_{M_1}^{\left(2\right)}\,\,\forall\overline{w}\in W'_{\ell} \right\} \subseteq V_{M_1}^{\left(2\right)} \,\,\text{.}
\]
Thus, by Lemma \ref{lem:deg-height--word-eqs} and Proposition \ref{prop:types},
\[
\left\{ \overline{g}\in \HH_{p}^{2}\mid\overline{w}\left(\overline{g}\right)\in V_{M_1,p}^{\left(2\right)} \,\,\forall\overline{w}\in W'_{\ell} \right\} \subseteq V_{M_1,p}^{\left(2\right)} 
\]
for every prime number $p$ outside a set $\calB^{\left(\ell\right)}$
of prime numbers such that 
\begin{align*}
\left|\calB^{\left(\ell\right)}\right|\ll_{n}\left|W_{\ell}\right|^{n^{2}+1}\ell^{O_{n}\left(1\right)}\ll_{n}2^{5n^{2}\ell}
\end{align*}
Thus,
\[
\PP_{\overline{w}\sim\Unif\left(W_{\ell}\right)}\left(\overline{w}\left(\overline{g}\right)\in V_{M_1,p}^{\left(2\right)} \right)\leq Ce^{-\ell/C}
\]
for every prime number $p\notin\calB^{\left(\ell\right)}$ and every $\overline{g}\in \HH_p^{2}\setminus V_{M_1,p}^{\left(2\right)} $.

Let $\calB_{\delta}=\bigcup_{k=1}^{\infty}\calB^{\left(\lfloor\delta k/\left(5n^{2}\right)\rfloor\right)}\cap[2^{k-1},2^{k})$.
Then for all $X\geq 1$,
\begin{align*}
\left|\left\{ p\in\calB_{\delta}\mid p<X\right\} \right| & \leq
\sum_{k=1}^{\lceil\log_{2}X\rceil}\left|\calB^{\left(\lfloor\delta k/\left(5n^{2}\right)\rfloor\right)}\right| \leq
\sum_{k=1}^{\lceil\log_{2}X\rceil}2^{\delta k}
\ll_{n}X^{\delta}\,\,\text{,}
\end{align*}
proving \eqref{eq:gaps-bad-primes-sparse}.

Take a prime number $p\notin\calB_{\delta}$. Take $k$ such that $p\in[2^{k-1},2^{k})$,
and let $\ell_{p}=\lfloor\delta k/\left(5n^{2}\right)\rfloor$. 
Then $\ell_{p}\gg_{n} \delta \log p$ and $\log p \ll_n \log |G|$ since $G$ has bounded index in $\widetilde{G}(\FF_p)$, which by the Lang--Weil estimate has cardinality $p^{\dim \widetilde{G}}(1+o_{n}(1))$. In particular, $\ell_p \geq \kappa_2 \log |G|$, where $\kappa_2 \ll_n \delta$, and as soon as $p\gg_{\delta,n} 1$, we get that for all  $\overline{g}=(g_1,g_2) \in \HH_p^2 \setminus V^{(2)}_{M_1,_p}$,
\begin{equation*}\PP_{\overline{w}\sim\Unif(W_{\ell_{p}})}\left(\overline{w}(\overline{g})\in V^{(2)}_{M_1,_p}\right)  \leq e^{-\ell_{p}/2C}.
\end{equation*}
By Lemma \ref{generating}, if $H$ is a proper subgroup of $G=G^+$ and  $w_1(\overline{g}),w_2(\overline{g})\in H$, then $\overline{w}(\overline{g}) \in V_{M_1,p}^{(2)}$. Hence for $\PP=\PP_{\overline{w}\sim \Unif (W_{\ell_p})}$:
\begin{equation*}\mu_{\overline{g}}^{\ell_p}(H)^2  = \PP(w_1(\overline{g}), w_2(\overline{g}) \in H)
 \leq \PP(\overline{w}(\overline{g}) \in  V_{M_1,p}^{(2)})
 \leq \exp(-\ell_p/2C).
\end{equation*}
This means that the non-concentration estimate $(\ref{noncon})$ holds with $c=1/4C$ and $\kappa_2 \ll_n \delta$ for every pair $\overline{g}=(g_1,g_2)$ that generates $G^+$.  We are done.

\begin{rem}\label{explicit}
Examining the proof, we can see that $\log \eps^{-1} \ll_{d,r} 1+\delta^{-1}$. 
\end{rem}

\subsection{Perfect groups}
It turns out that Theorem \ref{expmain} continues to hold for \emph{perfect} groups $\GG$, i.e. $[\GG,\GG]=\GG$. Perfectness is a necessary assumption for a uniform spectral gap, because large abelian quotients are obstructions to the property of being an expander. Thanks to recent work of Golsefidy-Srinivas \cite{golsefidy-srinivas}, expansion for perfect groups can be reduced to the semisimple case. We have:

\begin{thm}\label{perfectaddon}Let $\GG$ be a perfect connected algebraic group defined over $\QQ$ and $\delta>0$. There is $\eps>0$ and  a (possibly empty) set of prime numbers $\mathcal{B}_\delta$ satisfying \eqref{eq:gaps-bad-primes-sparse} such that for all primes $p\notin \mathcal{B}_\delta$  $\GG(\FF_p)$, is $\eps$-expanding.
\end{thm}

\begin{proof}This follows directly by combining \cite[Corollary F]{golsefidy-srinivas} and Corollary \ref{expansion}. In the notation of \cite{golsefidy-srinivas}, we have $\GG=\HH\ltimes \UU$, where $\UU$  is unipotent and $\HH$ is a semisimple $\QQ$-group, which without loss of generality we may assume to be simply connected (otherwise pass to a simply connected cover). Note that $\HH(\FF_p)$ is well-defined for all but finitely many primes $p$.
\end{proof}

\begin{rem}

 (i) The same proof gives uniform expansion for quotients of the form $\GG(\ZZ/q\ZZ)$ where $q$ is a product of a bounded number of distinct primes outside $\mathcal{B}_\delta$. See the formulation of \cite[Corollary F]{golsefidy-srinivas}, which now holds for all perfect $\GG$ thanks to Theorem \ref{perfectaddon}.

(ii) The use of \cite{golsefidy-srinivas} could be circumvented by using our method exclusively, but it would require to extend  \cite[Theorem 1.3]{becker-breuillard} to affine groups. This is possible, but is left for future work.
\end{rem}

\section{Proof of the dimension zero theorem}\label{dzerosec}
In this section we prove:

\begin{thm}(dimension zero theorem)\label{expmaindim0} For every  $d\ge 1$, and $m \ge 2$, there is $M_1\ge 1$ and a function $\eps=\eps(\delta,d,m)>0$ with the following property. If  $\GG_\ZZ$ is a simple Chevalley $\ZZ$-subscheme of $\GL_d$, $q$ is a power of a prime $p$ and $G\leq \GG_p(\FF_q)$ an $M_1$-dense subgroup, then for every $\delta>0$ all but possibly at most $O_{\delta,d,m}(q^\delta)$  $G$-conjugacy classes of $m$-tuples $(a_1,\ldots,a_m)\in G^m$ that generate $G$ are $\eps$-expanding. 
\end{thm}

Since every finite simple group of Lie type $G(q)$ is $M_1$-dense in some $\GG_p$ provided $q \gg_{M_1} 1$ (see e.g. \cite{larsen-pink} and the discussion in \cite[Section 5]{bggt}), Theorem \ref{dim0} is a special case of the above statement.

The proof will proceed in a similar way as that of Theorem \ref{cro-refor}, although several steps require a different execution. Since we are in positive characteristic, reduction modulo an irreducible polynomial of degree $d$ in $\FF_p[t]$ always leads to the same finite field $\FF_{p^d}$. Even though there are many irreducible polynomials of given degree in $\FF_p[t]$, this produces  too few different quotients. So we cannot hope to simply mimic the proof of Theorem \ref{cro-refor} using reduction modulo various primes: a new approach is required.  

Another important new obstacle here is the presence of subfield subgroups $\GG(\FF_{q'})$ for a subfield $\FF_{q'}$ of $\FF_q$. These are not Zariski-dense as they are finite, but they are $M$-dense for a fixed $M$ if $q'$ is large.


We already mentioned in Section \ref{sec-prime} that there is some finite $M=M_\GG>0$ such that every finite $M$-dense subgroup of $\GG$ is a subfield subgroup. In fact, we may choose $M$ so that also every infinite $M$-dense subgroup is Zariski-dense. Indeed, there is a finite number of irreducible $\GG$-modules with the property that every non Zariski-dense subgroup of $\GG$ must act non-irreducibly on one of them or be finite (see e.g. \cite[Lemma 4.2]{bggt-israel} and \cite[Theorem 11.7]{guralnick-tiep}). Not acting irreducibly is an algebraic condition of bounded degree. In what follows, we will assume that $M$ satisfies this condition.

\begin{rem}\label{explicit2} An explicit bound on $M$ can be derived following the arguments in \cite[Theorem 11.7]{guralnick-tiep} and the comments after \cite[Lemma 4.2]{bggt-israel} using an explicit bound on Larsen-Pink's Jordan constant. A bound for this constant has been derived in \cite{bajpai-dona}. 
\end{rem}

These hurdles are overcome by studying, for each collection $E$ of  pairs of words $w_1,w_2$ of length $\ell$, the subscheme $V_E$ of $\GG^m$ where these pairs do not generate an $M$-dense subgroup. We work over an algebraic closure of $\FF_p$ and consider $m$-tuples up to the diagonal action $g\cdot (g_1,\ldots,g_m) = (gg_1g^{-1},\ldots,gg_mg^{-1})$; that is, we work in the character variety $\GG^m/\!\!/\GG$ of the free group on $m$ letters with values in $\GG$. A first key observation (cf. Lemma \ref{lem:component-of-few-eqs-1}) is that the total number of irreducible subvarieties appearing as geometric components of some $V_E$ is only exponential in $\ell$ rather than super-exponential (as is the number of such collections $E$).  We are then in a position to apply our uniform anti-concentration bound from \cite{becker-breuillard} and show that these components must be zero-dimensional when the probability that a random pair $(w_1,w_2)$ lies in $E$ is larger than sub-exponential. Together this will force the desired upper bound on the number of bad $m$-tuples. 

We now pass to the proof.  We set $m=2$ for simplicity, the case of a general $m$ being entirely similar. 

\vspace{.5cm}

{\bf Step 1: Only exponentially many types.}
For a set $E$ of pairs of (non-reduced) words $(w_1,w_2)$ of length $\ell$ in two letters, consider the subscheme $V_E$ of $\GG^2$ defined by the equations $(w_1(a,b),w_2(a,b)) \in V_{M}^{(2)}$ for $(w_1,w_2) \in E$, where $V_{M}^{(2)}$ is as earlier (see Proposition \ref{degprop}) and defines pairs that generate a subgroup that is not $M$-dense in $\GG$. We work modulo conjugation by the diagonal action of $\GG$ on pairs, since the schemes we define are naturally invariant under diagonal conjugation. We will thus consider the scheme $\overline{V}_E$ (resp.  $\overline{V}_{M}^{(2)}$) in $\GG^2/\!\!/\GG$ associated to $V_E$ (resp.  $V_{M}^{(2)}$). Since word maps associated to words of length $\ell$ have degree at most $d\ell$ (see Lemma \ref{lem:deg-height--word-eqs}), each $V_E$ is defined by equations  in $r$ variables of degree at most $d\ell$  (taking matrix entries and the inverse of the determinant, this makes  $r \leq 2(d^2+1)$).

Since there are $4^\ell$ words $w(a,b)$ of length $\ell$,  Lemma \ref{lem:component-of-few-eqs-1}, implies an exponential bound on the number of geometrically irreducible varieties appearing as components of some $\overline{V}_E$. For brevity we will call a geometrically irreducible component of some $\overline{V}_E$ an \emph{$\ell$-type}. 

\begin{lem}\label{typesnb} There are at most $(d\ell 4^{2\ell})^r\leq 2^{5dr\ell} $ $\ell$-types.
\end{lem}

  This exponential bound is to be contrasted with the fact that there are at least $2^{2^{3\ell}}$ possible subsets $E$.  Given an $\ell$-type $\mathcal{C}$, we further denote by $E(\mathcal{C})$ the collection of pairs of words $(w_1,w_2)$ such that $\mathcal{C} \subset \overline{V}_{\{(w_1,w_2)\}}$. Note that $E \subset E(\mathcal{C})$. We observe further that if $(a,b)$ is a generic point of $\mathcal{C}$, then $(w_1,w_2) \in E(\mathcal{C})$ if and only if  $(w_1(a,b),w_2(a,b)) \in V_{M}^{(2)}$. 

\vspace{.5cm}

{\bf Step 2: Thin and fat types.}
The main theorem of \cite{becker-breuillard} is an anti-concentration bound on subvarieties for random walks on linear groups. It asserts that there is a constant $c=c(d,m)>0$ independent of the characteristic of the field, such that if $(a_1,\ldots,a_m)$ generates a Zariski-dense subgroup of $\GG$, then for every proper  closed algebraic subvariety $V$ in $\GG$,
$$\PP_{|w|=\ell}(w(a_1,\ldots,a_m) \in V) \ll_{\deg V} 2^{-c\ell},$$ 
where $\PP_{|w|=\ell}$ denotes the probability of an event when the (non-reduced) word $w$ is randomized uniformly among words of length $\ell$ with letters in $a_i^{\pm 1}$. As we did in \S \ref{subsec-proofmain}, we will apply this to $\GG \times \GG$ with $16$ letters $(g^{\pm 1}, h^{\pm 1})$, $g,h \in \{a,b\}$. The uniform measure on words of length $\ell$ in these letters replicates exactly the product measure on pairs $(w_1,w_2)$ of words of length $\ell$ in $(a,b)$. Also the $16$ pairs generate a Zariski-dense subgroup of $\GG \times \GG$ if and only if $(a,b)$ generates a Zariski-dense subgroup of $\GG$. And we will take the variety $V=V_{M}^{(2)}$. In particular, whenever $(a,b)$ generates a Zariski-dense subgroup of $\GG$ we have:
\begin{equation}\label{probnd}\PP_{|w_1|=|w_2|=\ell}( (w_1(a,b),w_2(a,b)) \in V_M^{(2)} )  \leq C_M 2^{-c_\GG\ell}\end{equation}
for a constant $c_\GG=c(\dim \GG)>0$, independent of the characteristic of the field and $C_M>0$ depends only on $M,d$. 

We will say that a subset $E$ of pairs of words $(w_1,w_2)$ is $M$-thin if $\PP_{|w_1|=|w_2|=\ell}(E) \leq C_M 2^{-c_\GG\ell}$ and $M$-fat otherwise. We say that an $\ell$-type $\mathcal{C}$ is $M$-thin (resp. $M$-fat) if $E(\mathcal{C})$ is $M$-thin (resp. $M$-fat). 

\begin{lem} If \eqref{probnd} fails for a pair $(a,b) \in \GG \times \GG$, then its diagonal conjugacy class in $\GG^2/\!\!/\GG$ lies in an $M$-fat $\ell$-type. 
\end{lem}

\begin{proof}Indeed, the set $E$ of pairs of words $(w_1,w_2)$ with $(w_1(a,b),w_2(a,b)) \in V_M^{(2)}$ is $M$-fat, and  $(a,b)$ lies in $V_E$. Its conjugacy class must then lie in some geometrically irreducible component $\mathcal{C}$ of $\overline{V}_E$, which must be $M$-fat because $E \subset E(\mathcal{C})$. 
\end{proof}

The crux of the proof of Theorem \ref{dim0} lies in the following lemma.

\begin{lem}\label{fat-dim0} If $\mathcal{C}$ is an $M$-fat $\ell$-type, then either  $\mathcal{C} \subset \overline{V}_M^{(2)}$ or  $\dim \mathcal{C}=0$.
\end{lem}

\begin{proof}Suppose $\mathcal{C}$ is a positive dimensional $\ell$-type. Pick $(a,b)\in \GG\times \GG$ whose conjugacy class yields a generic point of the variety $\mathcal{C}$. By \cite[Lemma 4.2]{bggt-israel}, and since $\GG$ is assumed simple, if the subgroup generated by $(a,b)$ is $M$-dense, then it is either Zariski-dense or it is finite and conjugates into $\GG(\FF_q)$ for some finite field $\FF_q$. However, since $\dim \mathcal{C}>0$, the generic point has an infinite orbit under the Galois action. So the group cannot be finite and it is therefore either Zariski-dense or not $M$-dense. In the latter case $(a,b) \in V_M^{(2)}$ and as $(a,b)$ is generic in $\mathcal{C}$ and $V_M$ is defined over $\FF_p$ we must have $\mathcal{C} \subset \overline{V}_M^{(2)}$. To exclude the former case, we apply the main anti-concentration result of \cite{becker-breuillard} and conclude that \eqref{probnd}  holds for $(a,b)$. But this means that $\mathcal{C}$ must be $M$-thin, contrary to our assumption. 
\end{proof}

\vspace{.5cm}

{\bf Step 3: Non-concentration on subgroups.}
We are now in a position to prove the required non-concentration on subgroups, which is necessary to kick start the Bourgain-Gamburd method.  To prove expansion for the pair $(a,b)$ we need to find $c,\kappa>0$ depending only on $\dim \GG$ and on $\delta>0$ such that 
\begin{equation}\label{noncon2} \PP_{|w|=\ell}(w(a,b) \in H) \leq q^{-\kappa}\end{equation}
for all $\ell \ge c \log q$ and all proper subgroups $H<G$. Under the assumptions of Theorem \ref{expmaindim0} it may happen that $G$ has subgroups of small index. However, there is a uniform bound on the index of such. In fact, by the main result of Larsen-Pink, \cite[Theorem 0.5]{larsen-pink}, the group $G$ contains the quasi-simple group $[\GG^F,\GG^F]$, and is contained in $\GG^F$ for a certain Frobenius map, i.e. an endomorphism of $\GG$ a power of which is $\Frob_{p^n}$ for some $n\ge 1$. The group $[\GG^F,\GG^F]$ has no non-central normal subgroups and its index in $\GG^F$ is bounded in terms of the rank (hence $d$) only. By the classical result of Landazuri-Seitz, \cite{landazuri-seitz} the quasi-simple groups  $[\GG^F,\GG^F]$ are $\beta$-quasirandom; this means that there is $\beta=\beta(d)>0$ such that every non-trivial complex representation has dimension at least $|\GG^F|^\beta$. It follows that there is a uniform bound, say $C(d)>0$ such that subgroups of $G$ of index at least $C(d)$ have index at least $|G|^\beta$. The Bourgain-Gamburd method (see e.g. \cite[Proposition 3.1]{bggt}) usually assumes that the group has no subgroup of small index, but it works just as well in the situation when these all have bounded index, see \cite[Proposition 3.1]{breuillard-standrews}. 
Therefore we are left to show \eqref{noncon} with the condition $[G:H]>1$ replaced by $[G:H]>C(d)$. 


Recall then that subgroups of $G$ of index at least $C(d)$ can be of two types only: \emph{structural subgroups}, namely subgroups contained in a proper algebraic subgroup of $\GG$ of bounded degree defined over $\FF_p^{\al}$, and \emph{subfield subgroups}, namely subgroups that can be conjugated into $\GG(\FF_{q'})$ for a certain subfield $\FF_{q'}$ of $\FF_q$. We refer the reader to \cite{bggt} and in particular the discussion around Lemma 5.5 there for more information on this fact. Structural subgroups are easier to deal with as we clearly have:

\begin{lem}[structural subgroups]\label{struct} If $\HH\leq \GG$ is a proper algebraic subgroup of degree at most $M$ and if the conjugacy class of $(a,b)$ does not belong to an $M$-fat $\ell$-type, then \eqref{noncon2} holds for $H\leq \HH \cap G$ if $\ell \ge 4\kappa c^{-1} \log q$ and $q^\kappa >C_M$.
\end{lem}

\begin{proof}Let $E$ be the set of pairs $(w_1,w_2)$ such that $w_1(a,b)$ and $w_2(a,b)$ lie in $H$. Since $\HH$ is not $M$-dense,  $(w_1(a,b),w_2(a,b)) \in V_M^{(2)}$ for $(w_1,w_2) \in E$. In particular $(a,b) \in V_E$. Then the conjugacy class of $(a,b)$ yields a point in $\GG^2/\!\!/\GG$ that lies in some geometrically irreducible component $\mathcal{C}$ of $\overline{V}_E$ and $E \subset E(\mathcal{C})$. By the assumption on $(a,b)$, $\mathcal{C}$ must be $\ell$-thin and thus $E$ is $\ell$-thin. So $\PP(E) \leq C_M2^{- c \ell}$ by \eqref{probnd}. This is $\leq q^{-\kappa/2}$ provided $\ell \ge 4\kappa/c \log q$ and $q^\kappa>C_M$. 
\end{proof}

\vspace{.5cm}

{\bf Step 4: Subfield subgroups.} Let $q=p^n$. In our general setting $G$ need not be a split Chevalley group over $\FF_q$, and can be of twisted Steinberg or Ree type.  In general $G$ has bounded index in the group of fixed points $\GG^F$ of some Frobenius map $F$. A proper subfield subgroup of $G$ of large index will be contained in (a conjugate of) $\GG^{F_0}$ where $F_0$ is another Frobenius map such that $F_0^m=F$ for some integer $m>1$.  

To handle these subgroups and show the desired non-concentration, we need to use a variant of the notion of $\ell$-type discussed earlier. For this we consider the cartesian product $\LL=\GG \times \GG$ and let $\pi_1,\pi_2$ be the coordinate projections. For any subset $E$ of (non-reduced) words of length $\ell$ in the free group $F_2$, we let $W_E$ be the subset of all pairs $a,b$ in $\LL$ such that $w(a,b) \in \Delta_\LL$, where $\Delta_\LL$ is the diagonal subgroup $\{(x,y) \in \GG \times \GG, x=y\}$.  A geometrically irreducible subvariety of $\LL^2/\!\!/\LL$ will be called a diagonal $\ell$-type of $\LL$ if it is a geometric component of $\overline{W}_E$ for some subset $E$ of words of length $\ell$. As earlier, Lemma \ref{lem:component-of-few-eqs-1} implies that there are at most $2^{10dr\ell}$ diagonal types. 

We say that a diagonal $\ell$-type is thin or fat according as the following holds or not:
\begin{equation}\label{probnd2}\PP_{|w|=\ell}(w(a,b) \in \Delta_\LL )  \leq C 2^{-c_\LL\ell}.\end{equation}
This is the analogue of \eqref{probnd} where $C$ and $c_\LL$ are constants that depend only on $\dim \GG$ and are provided by the anti-concentration bound from \cite{becker-breuillard} applied to the subvariety $\Delta_\LL$ in the semisimple algebraic group $\LL$.  Let $\overline{W}_M$ be the degenerate locus  in $\LL^2/\!\!/\LL$, that is the subvariety of (conjugacy classes of) pairs $(a,b)$ that generate a subgroup of $\LL$ that is not $M$-dense in $\LL$.  Here $M=M_\LL$ is chosen $>M_\GG$ so that every $M$-dense subgroup of $\LL$ with infinite projection to $\GG$ under $\pi_1$ and $\pi_2$ is Zariski-dense in $\LL$. As before the main result of \cite{becker-breuillard} implies:

\begin{lem}\label{fatfat} If $\mathcal{C}$ is a fat diagonal $\ell$-type, then either $\mathcal{C} \subset W_M$ or $\dim \pi_i(\mathcal{C})=0$ for at least one $i=1,2$. 
\end{lem}

\begin{proof}If the conclusion fails, we may pick a generic point of $\mathcal{C}$ outside $\overline{W}_M$ whose image under $\pi_i$ is a generic point of $\pi_i(\mathcal{C})$ for $i=1,2$. Let $(a,b) \in \GG^2$ be a point above this generic point. Then the conjugacy class of $(\pi_i(a),\pi_i(b))$ is not algebraic over $\FF_p$ and hence generates an infinite subgroup of $\GG$. The Zariski-closure of $\langle a,b \rangle$ is $M$-dense in $\LL$, hence is all of $\LL$. Then the main result of \cite{becker-breuillard} implies that $\mathcal{C}$ is thin.
\end{proof}


\begin{lem}[subfield subgroups]\label{subfield} Let $a,b \in G$ and suppose the (conjugacy class of the) pair $((a,a^{F_0}),(b,b^{F_0}))$ does not belong to a fat diagonal $\ell$-type, then \eqref{noncon2} holds for  $H =  G \cap \GG^{F_0}$ if  $\ell \ge 2\kappa/c_\LL \log q$.
\end{lem}

\begin{proof}Let $E$ be the set of words $w$ such that $w(a,b) \in H$. We thus have $w(a^{F_0},b^{F_0})=w(a,b)$. In particular $w((a,a^{F_0}),(b,b^{F_0})) \in \Delta_\LL$. So $((a,a^{F_0}),(b,b^{F_0})) \in W_E$. Then $((a,a^{F_0}),(b,b^{F_0}))$ yields a point in $\LL^2/\!\!/\LL$ that lies in some geometrically irreducible component $\mathcal{C}$ of $\overline{W}_E$ and $E \subset E(\mathcal{C})$. Since it belongs to no  fat diagonal $\ell$-type, $\mathcal{C}$ must be thin and thus $E$ is thin. So $\PP(E) \ll_{\dim \LL} 2^{-c_{\LL}\ell}$. This is $\leq q^{-\kappa}$ provided $\ell \ge 2\kappa/c_\LL \log q$ and $q$ is large enough.
\end{proof}

Further note that if $((a,a^{F_0}),(b,b^{F_0}))\in W_M$, while $\langle a,b \rangle=G$, then the subgroup of $\LL$ it generates must have $M$-dense projections to $\GG$ (both projection) and thus, by Goursat's lemma, must be contained in a proper diagonal algebraic subgroup of $\LL$. This means that there is $\alpha \in \Aut \GG$ such that $a^{F_0}=\alpha(a)$ and $b^{F_0}=\alpha(b)$. Hence $G \leq \{g \in \GG, g^{F_0}=\alpha(g)\}$. But this is easily seen to be impossible, because $G$ has bounded index in $\GG^F$, and $F$ is a proper power of $F_0$. From Lemma \ref{fatfat}, we conclude that if $(a,b)$ generates $G$, and the conjugacy class of $((a,a^{F_0}),(b,b^{F_0}))$ belongs to a fat diagonal $\ell$-type $\mathcal{C}$ for some $F_0$, then $\dim \pi_i(\mathcal{C})=0$ for at least one $i=1,2$. In other words  $\pi_i(\mathcal{C})$ is a singleton, which is the conjugacy class of either $(a,b)$ or $(a^{F_0},b^{F_0})$. In particular, there are only as many such conjugacy classes as there are diagonal $\ell$-types up to a choice of Frobenius map $F_0$. This makes $\ll 2^{10drl} \times \log(q)/\log(p)$ choices for the conjugacy class of $(a,b)$.

\vspace{.5cm}

{\bf Step 5: End of the proof.}
We can now complete the proof of Theorem \ref{expmaindim0}. As before, $\eps$-expansion for the pair $(a,b)\in G$ follows by the Bourgain-Gamburd method (as spelled out e.g. in \cite[Proposition 3.1]{bggt}) from the following three ingredients: quasi-randomness, the product theorem and non-concentration on subgroups. As recalled earlier the first two ingredients hold for $G$ (with the special provision made to allow for subgroups of bounded index in $G$). So the only thing to check is that the non-concentration \eqref{noncon2} holds uniformly for all subgroups of index at least $C(d)$ in $G$. Now this follows from the combination of Lemmas \ref{struct} and \ref{subfield} provided $(a,b)$ generates $G$, its conjugacy class does not belong to an $M$-fat $\ell$-type of $\GG^2/\!\!/\GG$, the conjugacy class of $((a,a^{F_0}),(b,b^{F_0}))$ (for each $F_0$) does not belong to a fat diagonal $\ell$-type of $\LL^2/\!\!/\LL$ and $\ell \ge 2\kappa c^{-1} \log q$, where $c=\min\{c_\GG,c_\LL\}$. By Lemma \ref{typesnb} and the discussion after Lemma \ref{subfield} this holds for all $(a,b)$ except perhaps for $\ll 2^{10drl} \times \log(q)/\log(p)$  conjugacy classes of pairs. Note additionally that it is enough to count $\GG$-conjugacy classes, because there can be at most $O_r(1)$  $G$-conjugacy classes of pairs in a given $\GG$-conjugacy class outside $V^{(2)}_{M_1}$.

Choose $\kappa=\delta c/50dr$ and $\ell \simeq 4\kappa c^{-1} \log q$. Then $2^{10dr\ell} \times \log(q)/\log(p) \leq q^{\delta}$. Therefore, the total number of possible exceptions to \eqref{noncon2} is at most  $q^{\delta}$ conjugacy classes of pairs $(a,b)$. Outside these exceptions we get $\eps$-expansion with some $\eps=\eps(\GG,\delta)>0$. This ends the proof of Theorem \ref{expmaindim0}.

\bigskip

\paragraph{\textbf{Acknowledgement.}}
The first author has received funding from the European Research Council (ERC) under the European Union's Horizon 2020 research and innovation programme (grant agreement No. 803711).
The second author acknowledges partial support from EPSRC grant UKRI1017. For the purpose of Open Access, the authors have applied a CC BY public copyright licence to any Author Accepted Manuscript  version arising from this submission.



\end{document}